\documentclass[a4paper,10pt]{article}

\usepackage[utf8]{inputenc}
\usepackage{authblk}
\usepackage{amsthm,amsmath,amssymb,latexsym,dsfont}

\usepackage{cite}

\usepackage{enumerate}

\theoremstyle{definition}

\newtheorem{theorem}{Theorem}[section]
\newtheorem{proposition}[theorem]{Proposition}
\newtheorem{lemma}[theorem]{Lemma}
\newtheorem{corollary}[theorem]{Corollary}
\newtheorem{remark}[theorem]{Remark}

\newcommand{\N}{\mathbb{N}}
\newcommand{\Z}{\mathbb{Z}}

\newcommand{\R}{\mathbb{R}}

\newcommand{\A}{\mathcal{A}}

\newcommand{\C}{\mathcal{C}}
\newcommand{\F}{\mathcal{F}}
\newcommand{\G}{\mathcal{G}}
\newcommand{\HH}{\mathcal{H}}

\newcommand{\Y}{\mathcal{Y}}

\renewcommand{\P}{\mathbb{P}}
\newcommand{\E}{\mathbb{E}}

\newcommand{\ind}{\mathds{1}}

\newcommand{\eps}{\varepsilon}

\DeclareMathOperator{\cov}{Cov}
\DeclareMathOperator{\var}{Var}

\usepackage[margin=30mm]{geometry}
\usepackage[hidelinks]{hyperref}

\hypersetup{pdftitle={Robust Nonparametric Methods for alpha-Mixing Sequences via Median-of-Means},
	pdfauthor={Attila Lovas and Marcell Tukora}}

\newcommand{\Rb}{\mathcal{R}}          
\newcommand{\Ker}{\mathcal{K}}         
\newcommand{\VaR}{\mathrm{VaR}}
\newcommand{\CVaR}{\mathrm{CVaR}}
\newcommand{\med}{\operatorname{median}}

\title{Robust Nonparametric Methods for $\alpha$-Mixing Sequences\\ via Median-of-Means\thanks{This research was supported by the J\'anos Bolyai Research Scholarship of the Hungarian Academy of Sciences. Additional support was provided by the Thematic Excellence Program 2021, funded by the National Research, Development and Innovation Office under the subprogram ``Artificial intelligence, large networks, data security: mathematical foundation and applications'', and by the National Research, Development and Innovation Office through OTKA grant K 143529.}}

\author[1,2]{Attila Lovas\thanks{Corresponding author. E-mail: \texttt{lovas@renyi.hu}}}
\author[3]{Marcell Tukora}

\affil[1]{HUN-REN Alfr\'ed R\'enyi Institute of Mathematics, Budapest, Hungary}
\affil[2]{Budapest University of Technology and Economics, Budapest, Hungary}
\affil[3]{HUN-REN Wigner Research Centre for Physics, Budapest, Hungary}

\date{October 10, 2026}

\begin{document}

\maketitle

\begin{abstract}
The median-of-means (MoM) estimator is a robust alternative to the sample mean for heavy-tailed or contaminated data. While its non-asymptotic properties are well understood for i.i.d.\ samples, much less is known for dependent sequences. We establish non-asymptotic deviation bounds for the MoM estimator computed from consecutive blocks of an identically distributed $\alpha$-mixing sequence, both for a single mean and uniformly over a class of functions (the median-of-means process). All bounds remain valid, with explicitly adjusted constants, when a sufficiently small fraction of the blocks is adversarially corrupted. As applications, we obtain deviation bounds for a median-of-means Nadaraya--Watson regression estimator, for the MoM kernel density estimator, and for empirical and Bahadur-type estimators of Value-at-Risk and Conditional Value-at-Risk. Finally, we show that the sampling of bulk materials, where the measurements form a spatially dependent random field, reduces to the one-dimensional setting if the samples are taken in well-separated groups.

\end{abstract}

\section{Introduction}\label{sec:intro}

The problem of estimating the expectation $\mu = \E[X]$ of a random variable $X$ from a finite sample $X_1, \ldots, X_n$ goes back to Jacob Bernoulli’s \emph{Ars Conjectandi} (1713), where he proved that the relative frequency of an event converges to its probability.
Later, in his work on the method of least squares, Gauss showed that the sample mean has the smallest mean squared error among the linear unbiased estimators of the expectation.

However, as pointed out by Lugosi and Mendelson~\cite{lugosi2019mean}, the mean squared error, $\E[(\widehat{\mu}_n - \mu)^2]$ can be a misleading measure of performance when the error $|\widehat{\mu}_n - \mu|$ is highly variable. In many applications, notably in machine learning, one needs estimators that are close to the true mean with high probability already for moderate sample sizes, not only in expectation. From this point of view, the \emph{median-of-means} estimator has much better deviation properties than the sample mean for heavy-tailed data.

The median-of-means estimator is obtained by partitioning the sample into equally sized blocks, computing the sample mean within each block, and taking the median of the resulting block means.
This procedure was first introduced by Nemirovsky and Yudin~\cite{yudin1983problem} in the context of stochastic convex optimization.
It was later rediscovered and popularized in the computer science and machine learning literature, notably by Jerrum, Valiant, and Vazirani~\cite{jerrum1986random}, by Alon, Matias, and Szegedy~\cite{alon1996space}, and in recent years by Lugosi and Mendelson~\cite{lugosi2019mean}; see also~\cite{bubeck2013bandits,devroye2016sub}.

Building on this idea, Humbert, Le Bars, and Minvielle~\cite{humbert2022robust} defined the MoM kernel density estimator (MoM-KDE) as the pointwise median of kernel density estimates computed on disjoint blocks, and proved finite-sample $L^\infty$ bounds that hold when fewer than half of the blocks contain outliers.
In parallel, Lugosi and Mendelson~\cite{lugosi2019risk} studied the classical regression and statistical learning problem, where the goal is to predict the value of a random variable based on the observation of another.
They compared candidate predictors through medians of blockwise risk differences (\emph{median-of-means tournaments}) and obtained regression procedures with sub-Gaussian-type guarantees under weak moment assumptions, where empirical risk minimization can fail.

Beyond the i.i.d.\ setting, medians of averages have been used in the Markov chain Monte Carlo literature, where the median of several independent runs yields fixed-precision guarantees of the form $\P(|\hat I-I|\le\eps)\ge1-\alpha$~\cite{niemiro2009fixed,latuszynski2011rigorous}. Lerasle and Oliveira~\cite{lerasle2011robust} showed by a coupling argument that their median-of-means-type estimators can also be used for stationary $\beta$- and $\phi$-mixing data. 

More recently, median-of-means-type M-estimation under $\alpha$-mixing was considered for high-dimensional sparse regression, using blocks separated in time~\cite{mittal2024adaptive}, and Mulumudi et al.~\cite{mulumudi2026generalization} studied learning with the Conditional Value-at-Risk (CVaR) under heavy tails, including generalization bounds for $\beta$-mixing data and a truncated median-of-means estimator robust to adversarial contamination, a median-of-means-type statistic was used for testing changes in the mean of high-dimensional heavy-tailed data, with an extension to temporally dependent data~\cite{li2026robust}

A median-of-means estimator for corrupted, time-correlated data was also analyzed for temporal-difference learning with corrupted Markovian data~\cite{maity2025adversarially}, and empirical risk minimization with heavy-tailed, $\beta$-mixing data was studied in~\cite{roy2021empirical}. 

The closely related work of Xu et al.~\cite{xu2026non} uses the same blocking technique to
construct median-of-means estimators for high-dimensional mean vectors and autocovariance
matrices of heavy-tailed time series. The present paper differs in two respects. First,
Xu et al.\ work with functional dependence measures, whereas our results are established
under $\alpha$-mixing; neither condition implies the other (a stationary AR(1) process
with Bernoulli innovations is not strongly mixing~\cite{andrews1984non}, yet its functional
dependence measure decays geometrically), so the two sets of results are complementary.
Second, beyond means and autocovariance matrices, we study the median-of-means process and
derive results for a broader range of statistical problems, including nonparametric
regression, kernel density estimation, risk measures and sampling from spatially dependent
random fields.

Financial returns, sensor data, and measurements on bulk materials are dependent, and they may be heavy tailed or contain gross errors. To our knowledge, no non-asymptotic bounds are available for the median-of-means estimator and the median-of-means process computed from consecutive blocks of a \emph{single} $\alpha$-mixing trajectory that hold under the sole moment assumption $\E|X|^r<\infty$ for some $r>1$, uniformly over function classes and with adversarially corrupted blocks.

\smallskip
We measure the dependence between distant parts of the sequence $X_1, X_2, \ldots$ by the \emph{$\alpha$-mixing} introduced by Rosenblatt~\cite{rosenblatt1956central}.
Among the classical mixing conditions ($\alpha$-, $\beta$-, $\phi$- and $\rho$-mixing), $\alpha$-mixing (strong mixing) is the weakest; see Bradley~\cite{Bradley2005} for a survey.
For completeness, we recall the relevant definitions.
Throughout this paper, all random variables are defined on a common probability space $(\Omega, \F, \P)$.
Given two sub-$\sigma$-algebras $\G, \HH \subset \F$, the \emph{measure of dependence} between them is defined as
\begin{equation}\label{eq:dep_meas}
	\alpha(\G, \HH)
	= \sup_{G \in \G,\, H \in \HH}
	\left| \P(G \cap H) - \P(G)\P(H) \right|.
\end{equation}
For a stochastic process $(X_n)_{n \in \N}$, we define the associated sequence of \emph{mixing coefficients} by
\begin{equation}\label{eq:mix_coef}
	\alpha^X(n)
	:= \sup_{k \in \N}
	\alpha\!\left(
	\sigma(X_j : j \le k),\,
	\sigma(X_j : j \ge k + n)
	\right).
\end{equation}
Here and below $\N=\{0,1,2,\ldots\}$. The sequence $(\alpha^X(n))_{n \in \N}$ is clearly nonincreasing and bounded, and the process $(X_n)_{n\in\N}$ is said to be \emph{strongly mixing} (or $\alpha$-mixing) if $\alpha^X(n) \to 0$ as $n \to \infty$.

Under suitable moment and decay conditions on $\alpha^X(n)$, strong mixing provides a framework for extending classical probabilistic limit theorems beyond the i.i.d.\ setting, including the strong law of large numbers of McLeish~\cite{mcleish1975maximal}, the $L^1$-law of large numbers of Hansen~\cite{hansen2019weak}, and the central limit theorem of White~\cite{white2014asymptotic}.
Typically, such results require that $\alpha^X(n)$ converges to zero sufficiently fast and that $\sup_{n \in \N} \E|X_n|^r < \infty$ for some sufficiently large exponent $r>1$.

The mixing coefficients of an observed process are rarely known, and estimating them from a single trajectory is difficult.
McDonald, Shalizi, and Schervish~\cite{mcdonald2011estimating, mcdonald2015estimating} studied the histogram-based estimation of $\beta$-mixing coefficients. More recently, Khaleghi and Lugosi~\cite{khaleghi2023inferring} established strong consistency results for certain estimators of $\alpha$- and $\beta$-mixing coefficients for stationary ergodic processes.

In practice, one therefore relies on known decay rates for specific models. For example, stationary geometrically ergodic Harris recurrent Markov chains satisfy $\alpha(n)=O(\rho^n)$ for some $\rho\in(0,1)$; see, e.g., Section~2.4 of Doukhan~\cite{doukhan1995mixing} or Meyn and Tweedie~\cite{meyn2012markov}. 
Decay rates for autoregressive (AR), autoregressive moving-average (ARMA), and nonlinear autoregressive processes with exogenous inputs (ARX) are summarized in Doukhan~\cite{doukhan1995mixing}; for generalized autoregressive conditionally heteroskedastic (GARCH) and stochastic volatility models, see Carrasco and Chen~\cite{carrasco2002mixing}.
Recent papers \cite{lovas2024transition,lovas2025mixing} investigate random iterations driven by exogenous inputs and establish upper bounds on the $\alpha$-mixing coefficients of the resulting processes in terms of the $\alpha$-mixing coefficients of the input process. These results apply to a broad range of models arising in econometrics, queueing theory, and machine learning.

\smallskip
In this work, we demonstrate how the median-of-means estimator can be employed to robustly handle $\alpha$-mixing sequences. In particular, we show that this approach enables reliable estimation in several key statistical tasks, including nonparametric regression, kernel density estimation, and empirical and Bahadur-type estimation of risk measures such as Value-at-Risk (VaR) and CVaR, and that it extends to spatially dependent measurements, such as samples taken from a heap of bulk material.

\smallskip
Our main results are as follows.
\begin{itemize}
	\item We prove non-asymptotic deviation bounds for the median-of-means estimator of
	the mean of an identically distributed $\alpha$-mixing sequence with polynomially decaying mixing coefficients under the sole moment assumption $\E|X_0|^r<\infty$, $r>1$ (Theorem~\ref{thm:main}, Corollary~\ref{cor:specmix}). For every fixed tolerance $\delta>0$ and every fixed block size $B$ above an explicit threshold, the bounds decay as $O(n^{-a}\log^a n)$ under polynomial ($a>1$) and as $O(\kappa^{\sqrt n})$ under geometric mixing.

	\item We extend these bounds uniformly over a class of functions, i.e.\ to the median-of-means process (Theorem~\ref{thm:MoMproc}). We give a separability condition~\eqref{eq:S} under which the single-block suprema and, for continuous median rules, the supremum of the median-of-means process are random variables (Lemma~\ref{lem:meas}), and show that it is not implied by the Vapnik--Chervonenkis (VC) property (Remark~\ref{rem:VC}). The bounds remain valid when up to $n_o$ of the $K$ blocks are arbitrarily (even adversarially) corrupted.
	
	\item We apply the results to nonparametric regression with heavy-tailed responses (the MoM Nadaraya--Watson estimator, Section~\ref{sec:regression}), to kernel density estimation, extending the results of~\cite{humbert2022robust} to $\alpha$-mixing data (Section~\ref{sec:KDE}), and to VaR and CVaR (Section~\ref{sec:VaR}), where we provide explicit bounds on the probability that the estimation error exceeds a prescribed tolerance, both for empirical and for Bahadur-type estimators.

	\item We show that the sampling of bulk materials, such as heaps of cement or ore, where the measurements form a spatially dependent random field rather than a sequence, reduces to the one-dimensional setting (Section~\ref{sec:spatial}). 
\end{itemize}

\paragraph{Formal verification.} 
The main arguments of Sections~\ref{sec:main}--\ref{sec:spatial} have been formalized in the Lean~4 proof assistant on top of the Mathlib library: the median counting arguments, the measurability lemma, the transfer of mixing coefficients to blocks, the covariance inequality for bounded strongly mixing variables, the single-block bounds, Theorem~\ref{thm:MoMproc} with contamination, the deviation bounds of Sections~\ref{sec:regression}--\ref{sec:VaR} (with the single-block probabilities as hypotheses, which are supplied by separate formal lemmas), Proposition~\ref{prop:serial} and Corollary~\ref{cor:spatial}. Not formalized are the probabilistic part of the proof of Lemma~\ref{lem:Hansen} (hence of Theorem~\ref{thm:main} and Corollary~\ref{cor:specmix}, of which only the constants are checked), the identification of $G(q_\alpha)$ with $\CVaR_\alpha$ in Lemma~\ref{lem:RU}, and the asymptotic statements (Corollaries~\ref{cor:NW} and~\ref{cor:KDE}, Proposition~\ref{prop:KDEL1}(ii)). Within the formalized part, the only external input is Rio's inequality (Lemma~\ref{lem:Rio61}), which enters the formal statements as an explicit hypothesis; a polynomial version of Theorem~\ref{thm:MoMproc} avoiding it (Remark~\ref{rem:decrate}) is verified completely.

\section{Main results}\label{sec:main}

Let $(X_n)_{n\in\N}$ be a sequence of $\R$-valued identically distributed variables with $\E[X_0]=0$ for simplicity, and for a fixed integer $B\ge 1$, define the process of block averages:
\begin{equation}\label{eq:movAvg}
	\mu_k = \frac{1}{B}\sum_{l=1}^{B} X_{kB+l}, \quad k\in\N.
\end{equation}

Assume that $n=K\cdot B$ for some integer $K\ge1$. The \emph{median-of-means (MoM)} estimator is
\begin{equation}\label{eq:MoM}
	\hat{\mu} = \med(\mu_0,\ldots,\mu_{K-1}).
\end{equation}
Here and below, $\med(x_0,\ldots,x_{K-1})$ denotes any (measurable) choice of a median, that is, any $m\in\R$ with
\[
\#\{k: x_k\le m\}\ge K/2
\quad\text{and}\quad
\#\{k: x_k\ge m\}\ge K/2,
\]
e.g.\ the lower median $\min\{m: \#\{k: x_k\le m\}\ge K/2\}$. All results below hold for every such choice.
The aim of this section is to establish the main result of the paper,
namely, a counterpart of the well-known concentration inequality for the MoM estimator in the i.i.d.\ setting,
extended here to the case where $\E|X_0|^r<\infty$ for some $r>1$ and $(X_n)_{n\in\N}$ is merely $\alpha$-mixing with polynomially decaying coefficients.

\smallskip 
The following inequality provides a non-asymptotic exponential bound on the deviation of the normalized partial sums of a weakly dependent binary sequence.
This result serves as a replacement for Hoeffding's inequality, which is typically used in the proof of the concentration bound for the median-of-means estimator in the i.i.d.\ case.
\begin{lemma}\label{lem:Rio61}
	Let $(Z_n)_{n\in\N}$ be a sequence of $\{0,1\}$-valued random variables, and define
	$S_n = \sum_{k=1}^n \bigl(Z_k - \E[Z_k]\bigr)$, $n \ge 1$.
	Then, for any integers $1 < q < n$ and any $\delta \ge \frac{7q}{2n}$, we have
	\[
	\P\!\left( \frac{1}{n} |S_n| \ge \delta \right)
	\le
	4 \exp\!\left( -\frac{\delta n}{7q}
	\log\!\left( 1 + \frac{(2/7)\delta n}{\,n+q\,} \right) \right)
	+ \frac{28}{\delta}\,\alpha^Z(q+1).
	\]
\end{lemma}
\begin{proof}
	Let $1< q<n$ be fixed, and define $Y_k := Z_k\ind_{\{k\le n\}}$, $k\in\N$.
	We then apply Theorem~6.1 of Rio~\cite{rio2013inequalities} to the sequence of random variables $(Y_k)_{k\in\N}$, whose mixing coefficients do not exceed those of $(Z_k)$; in Rio's notation, $M=1$ and $M(n)=\sum_k\|Y_k\|_\infty\le n$.

	Note that $\|Y_k\|_{\infty}\le 1$ for every $k\ge1$, and $Y_k=0$ for $k>n$.
	For $v_q:=q(n+q)$ we have
	\[
	v_q \ge \Big\lceil\frac{n}{q}\Big\rceil q^2
	\ge \sum_{i>0}\E\big[(Y_{iq-q+1}+\cdots+Y_{iq})^2\big].
	\]
	Hence, by Theorem~6.1 of Rio~\cite{rio2013inequalities}, for any $\lambda\ge q$,
	\[
	\P\!\left(\sup_{1\le k\le n}|S_k|\ge \tfrac{7}{2}\lambda\right)
	\le 4\exp\!\left(-\frac{\lambda}{2q}\log\!\Big(1+\frac{\lambda}{n+q}\Big)\right)
	+ \frac{8n}{\lambda}\,\alpha^Z(q+1),
	\]
	where we have taken into account that in~\cite{rio2013inequalities}
	the definition of the mixing coefficient differs from ours by a factor of~$2$.
	Substituting $\lambda_n = \tfrac{2}{7}\delta n$ and assuming $\delta \ge \tfrac{7q}{2n}$,
	we obtain the desired inequality, which completes the proof.
\end{proof}

For later reference we denote the right-hand side of Lemma~\ref{lem:Rio61} with $n=K$ by
\begin{equation}\label{eq:Rbound}
	\Rb_{K,q}(\Delta,t)
	:=
	4 \exp\!\left(
	-\frac{\Delta K}{7q}
	\log\!\Bigl(1+\frac{(2/7)\Delta K}{K+q}\Bigr)
	\right)
	+ \frac{28}{\Delta}\,t,
	\qquad \Delta>0,\ t\ge0 .
\end{equation}
Note that $\Rb_{K,q}(\Delta,t)$ is nondecreasing in $t$ and nonincreasing in $\Delta$.

\begin{lemma}\label{lem:Hansen}
	Let $(X_n)_{n\in\N}$ be an $\R$-valued sequence of identically distributed variables satisfying $\E[X_0]=0$ and $\E|X_0|^r<\infty$ for some $r>1$, and define
	\[
	S_n = \sum_{k=1}^n X_k, \quad n\ge 1.
	\]
	Assume further that the strong mixing coefficients of $(X_n)_{n\in\N}$ satisfy $\alpha^X(n) = O(n^{-a})$ for some $a>0$, $a\neq 1$.
	Then there exists a constant $c>0$, depending only on $r$, $\E|X_0|^r$ and the sequence $(\alpha^X(n))_{n\in\N}$, such that for all $\delta>0$,
	\[
	\P\!\left(\frac{1}{n}|S_n|\ge \delta\right)
	\le c\,\delta^{-1}\,n^{-\frac{1}{2}\bigl(1-(1-a)_+\bigr)\left(1-\frac{1}{r}\right)}.
	\]
	Note that $1-(1-a)_+=\min(a,1)$. The same bound holds, with the same constant, if $S_n$ is replaced by $X_{m+1}+\cdots+X_{m+n}$ for any $m\in\N$, since the proof only uses the common marginal law and the mixing coefficients.
\end{lemma}
\begin{proof}
Following the truncation technique employed by Hansen~\cite{hansen2019weak},
let $n\ge 1$ and $\tau>0$ be arbitrary and fixed.
Define the truncated random variables
\[
X_k' = X_k\ind_{\{|X_k|< n^\tau\}} - \E[X_k\ind_{\{|X_k|< n^\tau\}}],
\]
and set $X_k'' = X_k - X_k'$, $k\in\N$.
By the triangle inequality, we arrive at
\begin{equation}\label{eq:decomp}
	\E\!\left[\frac{1}{n}|S_n|\right]
	\le \E\!\left[\frac{1}{n}\left|\sum_{k=1}^n X_k'\right|\right]
	+ \frac{1}{n}\sum_{k=1}^n \E|X_k''|,
\end{equation}
where the second term can be estimated further by using
\begin{align*}
	\E|X_k''|
	&\le 2\,\E[|X_k|\ind_{\{|X_k|\ge n^\tau\}}]
	\le \frac{2\,\E[|X_0|^r]}{n^{\tau(r-1)}}.
\end{align*}

Furthermore, $X_k'$ satisfies the bound $|X_k'|\le 2n^\tau$, and its mixing coefficients obey
$\alpha^{X'}(m)\le \alpha^{X}(m)$ for all $m\in\N$.
Using the well-known covariance inequality for bounded strongly mixing variables
(cf.\ Theorem~A.5 in~\cite{hall2014martingale} or the proof of Lemma~A.1 in~\cite{lovasCLT}), we obtain
\[
|\E[X_k' X_l']|
= |\cov(X_k', X_l')|
\le 16\,n^{2\tau}\,\alpha^{X}(|k-l|),\quad 1\le k,l\le n.
\]

Thus, by Jensen's inequality, we can write
\[
\E^2\!\left|\sum_{k=1}^n X_k'\right|
\le \sum_{k,l=1}^n \E[X_k' X_l']
\le 16\,n^{2\tau}\sum_{k,l=1}^n \alpha^X(|k-l|)
\le 32\,n^{2\tau+1}\sum_{k=0}^{n-1}\alpha^X(k).
\]
Since $\alpha^X(n)=O(n^{-a})$, for $a\neq 1$ there exists a constant $c'>0$ such that
\[
\sum_{k=0}^{n-1}\alpha^X(k)\le c'\,n^{(1-a)_+}, \qquad n\ge 1.
\]
Substituting this estimate into~\eqref{eq:decomp}, we obtain
\begin{equation}\label{eq:final-bound}
	\E\!\left[\frac{1}{n}|S_n|\right]
	\le \sqrt{32c'}\,n^{\,\tau-\frac{1}{2}+\frac{(1-a)_+}{2}}
	+ 2\,\E[|X_0|^r]\,n^{-\tau(r-1)}.
\end{equation}
Finally, for
$\tau=\frac{1}{2r}\bigl(1-(1-a)_+\bigr),$
both exponents in~\eqref{eq:final-bound} are equal to $-\frac{1}{2}\bigl(1-(1-a)_+\bigr)\left(1-\frac{1}{r}\right)$, and we arrive at
\[
\E\!\left[\frac{1}{n}|S_n|\right]
\le \bigl(\sqrt{32c'}+2\,\E[|X_0|^r]\bigr)
n^{-\frac{1}{2}\bigl(1-(1-a)_+\bigr)\left(1-\frac{1}{r}\right)},
\]
from which the claimed inequality follows immediately by Markov's inequality.
\end{proof}

\begin{remark}
	In the case $a=1$, that is, when $\alpha^X(n) = O(n^{-1})$,
	the above argument can be carried out in the same way.
	The only difference is that $\sum_{k=0}^{n-1}\alpha^X(k)\le c'(1+\log n)$,
	and therefore the bound in~\eqref{eq:final-bound} takes the form
	\[
	\E\!\left[\frac{1}{n}|S_n|\right]
	\le \sqrt{32c'}\,n^{\,\tau-\frac{1}{2}}\sqrt{1+\log n}
	+ 2\,\E[|X_0|^r]\,n^{-\tau(r-1)}.
	\]
	Choosing $\tau=\frac{1}{2r}$ then yields
	\[
	\P\!\left(\frac{1}{n}|S_n|\ge \delta\right)
	\le c\,\delta^{-1}\,
	n^{-\frac{1}{2}\left(1-\frac{1}{r}\right)}\sqrt{1+\log n},
	\]
	where $c=\sqrt{32c'}+2\,\E[|X_0|^r]$.

	Since this case does not contribute significantly to the main results of the paper,
	we shall not discuss it further and will always assume that $a\neq 1$ in what follows.
\end{remark}

The proof of the following theorem follows the argument of Chen~\cite{chen2020short} for independent data, with Hoeffding's inequality replaced by Lemma~\ref{lem:Rio61}.

\begin{theorem}\label{thm:main}
	Let $(X_n)_{n\in\N}$ be a sequence of $\R$-valued identically distributed variables such that $\E[X_0]=0$ and $\E|X_0|^r<\infty$ for some $r>1$.
	Assume that the strong mixing coefficients satisfy $\alpha^X(n)=O(n^{-a})$ for some $a>0$, $a\neq 1$.
	Let $c>0$ be the constant of Lemma~\ref{lem:Hansen} (depending only on $r$, $\E|X_0|^r$ and $(\alpha^X(n))_{n\in\N}$), and set
	\begin{equation}\label{eq:Bsize}
	b= \frac{1}{2}[1-(1-a)_+](1-1/r),
	\qquad
	C=(2c)^{1/b}.
	\end{equation}
	Let $\delta>0$, and suppose that the block size $B$ in the block-average process~\eqref{eq:movAvg} satisfies $B > C\,\delta^{-1/b}$. Then
	\[
	\Delta
	:= \frac{1}{2}
	- \frac{c}{\delta}\,
	B^{-b}
	= \frac12\left(1-\Bigl(\frac{C\delta^{-1/b}}{B}\Bigr)^{b}\right)>0,
	\]
	and for every integer $1<q\le\tfrac{2}{7}K\Delta$ the median-of-means estimator $\hat{\mu}$ defined in~\eqref{eq:MoM} obeys
	\[
	\P(|\hat{\mu}|\ge \delta)
	\le \Rb_{K,q}\bigl(\Delta,\alpha^X(qB)\bigr)
	= 4 \exp\!\left(
	-\frac{\Delta K}{7q}
	\log\!\Bigl(1+\frac{(2/7)\Delta K}{K+q}\Bigr)
	\right)
	+ \frac{28}{\Delta}\,\alpha^X(qB).
	\]
\end{theorem}

\begin{proof}
	As in the classical i.i.d.\ case, the key observation is that the event $\{|\hat{\mu}|\ge \delta\}$ implies that at least half of the block averages $\mu_0,\ldots,\mu_{K-1}$ deviate from zero by at least~$\delta$. Indeed, if $\hat\mu\ge\delta$, $\mu_l\ge\hat\mu\ge\delta$ for at least $K/2$ indices $l$, and the case $\hat\mu\le-\delta$ is symmetric. Hence,
	\begin{align*}
		\P(|\hat{\mu}|\ge \delta)
		&\le
		\P\!\left(\sum_{l=0}^{K-1} \ind_{\{|\mu_l|\ge \delta\}} \ge \frac{K}{2}\right)
		\le
		\P\!\left(\frac{1}{K}\Big|\sum_{l=0}^{K-1} (Z_l-\E Z_l)\Big|
		\ge \frac{1}{2}-\frac{1}{K}\sum_{l=0}^{K-1}\E Z_l\right),
	\end{align*}
	where $Z_l=\ind_{\{|\mu_l|\ge \delta\}}$, $0\le l<K$.

	By Lemma~\ref{lem:Hansen}, applied to the $l$-th block, there exists a constant $c>0$, depending only on $r$, $\E|X_0|^r$ and $(\alpha^X(n))_{n\in\N}$, such that for every $0\le l<K$,
	\[
	\E Z_l = \P(|\mu_l|\ge \delta)
	\le c\,\delta^{-1} B^{-b}.
	\]

	Since
	$B > C\,\delta^{-1/b},$
	it follows that
	\[
	0<\Delta = \frac{1}{2}
	- \frac{c}{\delta}\,
	B^{-b} \le \frac{1}{2}-\frac{1}{K}\sum_{l=0}^{K-1}\E Z_l,
	\]
	and therefore
	\[
	\P(|\hat{\mu}|\ge \delta)
	\le
	\P\!\left(
	\frac{1}{K}\Big|\sum_{l=0}^{K-1}(Z_l-\E Z_l)\Big|
	\ge \Delta
	\right).
	\]

	Applying Lemma~\ref{lem:Rio61} to the $\{0,1\}$-valued sequence $(Z_l)_{0\le l<K}$ yields, for every integer $1<q\le \tfrac{2}{7}K\Delta$,
	\[
	\P\!\left(
	\frac{1}{K}\Big|\sum_{l=0}^{K-1}(Z_l-\E Z_l)\Big|
	\ge \Delta
	\right)
	\le
	4 \exp\!\left(
	-\frac{\Delta K}{7q}
	\log\!\Bigl(1+\frac{(2/7)\Delta K}{K+q}\Bigr)
	\right)
	+ \frac{28}{\Delta}\,\alpha^Z(q+1).
	\]
	(The condition $q<K$ of Lemma~\ref{lem:Rio61} holds since $\Delta\le 1/2$.)

	Finally, since $Z_l$ is a function of the $l$-th block, the $\alpha$-mixing coefficients of $(X_n)_{n\in\N}$ and $(Z_l)_{0\le l<K}$ satisfy
	\[
	\alpha^Z(q)\le \alpha^X\bigl((q-1)B+1\bigr)\le\alpha^X\bigl((q-1)B\bigr), \quad q\ge1,
	\]
	and the desired inequality follows.
\end{proof}

The following corollary describes, for a fixed tolerance level $\delta>0$,
how the deviation probability $\P(|\hat{\mu}|\ge\delta)$ decays as a function of the total sample size $n=K B$,
depending on whether the mixing coefficients $\alpha^X(n)$ decrease at a polynomial or exponential rate.

\begin{corollary}\label{cor:specmix}
	Assume that the sequence satisfies $\alpha^X(n) = O(n^{-a})$ with $a>1$.
	Then, as follows from the proof of Lemma~\ref{lem:Hansen}, one can take $C = (2c)^{1/b}$, where
	$c = \sqrt{32 c'} + 2\E|X_0|^r$ and $c' = \sum_{k=0}^{\infty} \alpha^X(k) < \infty$, while
	$b = \tfrac{1}{2}\!\left(1 - \tfrac{1}{r}\right)$.

	Since $n = K B$, for any fixed $\delta > 0$ and any fixed block size $B > C\delta^{-1/b}$, the following bounds hold as $K\to\infty$:
	\begin{itemize}
		\item Choosing $q = \lfloor c_\Delta\, K \log^{-1}(K B) \rfloor$ with $c_\Delta:=\frac{\Delta\log(1+\Delta/7)}{7a}$, we get
		\[
		\P(|\hat{\mu}|\ge \delta) = O\!\left(n^{-a}\log^a n\right).
		\]
		Indeed, since $q\le K$ we have $\log\bigl(1+\frac{(2/7)\Delta K}{K+q}\bigr)\ge\log(1+\Delta/7)$, hence the exponential term of $\Rb_{K,q}$ is at most $4n^{-a}$, while $\alpha^X(qB)=O\bigl((n/\log n)^{-a}\bigr)$.

		\item If, in addition, $\alpha^X(n) = O(\rho^n)$ for some $\rho \in (0,1)$, then for
		$q = \lfloor \sqrt{K/B} \rfloor$ one obtains
		\[
		\P(|\hat{\mu}|\ge \delta) = O\!\left(\kappa^{\sqrt{n}}\right)
		\]
		for some $\kappa \in (0,1)$.
	\end{itemize}
\end{corollary}

\subsection{The median-of-means process}

To extend the applicability of the above results, we consider the \emph{median-of-means process}. Let $(\Y,\A)$ be a measurable space, and let $(Y_n)_{n\in\N}$ be an identically distributed (not necessarily independent) sequence of $\Y$-valued random variables. Let $\C$ be a class of $\A$-measurable functions $f:\Y\to\R$ such that
\[
\E\big[|f(Y_0)|\big]<\infty, \quad f\in\C.
\]

For a fixed $f\in\C$, define
\[
X_n^f := f(Y_n)-\E[f(Y_n)], \quad n\in\N.
\]
Fix an integer block size $B\ge 1$, and define the block means $\mu_k(f)$, $k\in\N$, analogously to \eqref{eq:movAvg}. If the number of blocks is $K\ge 1$ (so that $n=K\cdot B$), we define, in analogy with \eqref{eq:MoM}, the median-of-means estimator
\[
\hat{\mu}(f):=\med\big(\mu_0(f),\ldots,\mu_{K-1}(f)\big).
\]
The mapping $f\mapsto \hat{\mu}(f)$ defines the \emph{median-of-means process} over $\C$.

In the remainder of this section, we extend known results for the median-of-means process from independent to $\alpha$-mixing sequences. Our goal is to derive non-asymptotic bounds for the $L^\infty$ loss over $\C$, that is, $\sup_{f\in\C} \big|\hat{\mu}(f)\big|$, under minimal assumptions. Teng et al. \cite{math11173730} bounded $
\P\!\left(\sup_{f\in\C} \big|\hat{\mu}(f)\big|\ge \delta\right)$, $\delta>0$, by tail bounds for the binomial distribution, which improve on the bounds based on Hoeffding's inequality. This requires independent blocks; Instead, as throughout the paper, we follow the approach of \cite{chen2020short} and rely on Lemma~\ref{lem:Rio61}.

\paragraph{Measurability.}
If $\C$ is uncountable, the supremum $\sup_{f\in\C}|\mu_l(f)|$ of uncountably many random variables need not be measurable. Then the indicators $Z_l=\ind_{\{\sup_{f\in\C}|\mu_l(f)|\ge\delta\}}$, which replace the Bernoulli variables in the proof of Theorem~\ref{thm:main}; their expectations and mixing coefficients are then undefined, and Lemma~\ref{lem:Rio61} does not apply. The measurability of such indicators is taken for granted in the derivation of the uniform bound in~\cite{chen2020short}, where Hoeffding's inequality is applied to them, and in the proof of Proposition~1 of~\cite{humbert2022robust}. In empirical process theory this difficulty is usually removed by a \emph{pointwise measurability} assumption~\cite[Example~2.3.4]{vandervaart1996weak}; see also the permissible classes of Pollard~\cite{pollard1984convergence} and the image admissible Suslin classes of Dudley~\cite{dudley1999uniform}. We need the following variant, which also controls the centering constants:
\begin{equation}\label{eq:S}
	\tag{S}
	\begin{minipage}{0.84\textwidth}
	there is a countable subclass $\C_0\subseteq\C$ such that for every $f\in\C$ there are $g_1,g_2,\ldots\in\C_0$ with
	$g_m(y)\to f(y)$ for every $y\in\Y$ and $\E[g_m(Y_0)]\to\E[f(Y_0)]$ as $m\to\infty$.
	\end{minipage}
\end{equation}
Condition~\eqref{eq:S} holds in each of the following situations:
\begin{enumerate}[(a)]
	\item $\C$ is countable;
	\item $\C$ contains a countable subset $\C_0$ which is dense with respect to the uniform norm $\|f\|_\infty=\sup_{y\in\Y}|f(y)|$ (the convergence of the means follows from $|\E g(Y_0)-\E f(Y_0)|\le\|g-f\|_\infty$);
	\item $\C=\{f_\theta:\theta\in\Theta\}$, where $\Theta$ is a separable metric space, $\theta\mapsto f_\theta(y)$ is continuous for every $y\in\Y$ and $\theta\mapsto\E f_\theta(Y_0)$ is continuous; the latter holds, e.g.\ by dominated convergence, if $\sup_\theta|f_\theta|\le F$ with $\E F(Y_0)<\infty$.
\end{enumerate}
All classes used in Sections~\ref{sec:regression}--\ref{sec:VaR} are of type (c).

\begin{theorem}\label{thm:MoMproc}
	Under the above assumptions, suppose that $\C$ satisfies~\eqref{eq:S}, and that for some $\delta>0$ and block size $B\ge 1$ there is a uniform upper bound $\eps_{\delta,B}$ such that
	\[
	\sup_{l\in\N}
	\P\!\left(
	\sup_{f\in\C} \big|\mu_l(f)\big|\ge \delta
	\right)
	\le \eps_{\delta,B},
	\]
	the events being measurable by Lemma~\ref{lem:meas} below. If
	\[
	0<\Delta:=\tfrac12-\eps_{\delta,B},
	\]
	then for every integer $1<q\le \tfrac{2}{7}K\Delta$,
	\[
	\P\!\left(\sup_{f\in\C} \big|\widehat{\mu}(f)\big|\ge \delta\right)
	\le
	\Rb_{K,q}\bigl(\Delta,\alpha^{Y}(qB)\bigr)
	=
	4 \exp\!\left(
	-\frac{\Delta K}{7q}
	\log\!\Bigl(1+\frac{(2/7)\Delta K}{K+q}\Bigr)
	\right)
	+\frac{28}{\Delta}\,\alpha^{Y}(qB).
	\]
	If the median is chosen by a rule that is continuous in the block means (e.g.\ the lower median), the event on the left-hand side is measurable (Lemma~\ref{lem:meas}); for an arbitrary choice of the median, $\P$ on the left-hand side is to be understood as outer probability.
\end{theorem}

Note that the mixing coefficients of the underlying sequence $(Y_n)$ appear on the right-hand side; those of $(X^f_n)$ depend on $f$ and are all dominated by $\alpha^Y$.

\begin{lemma}[Measurability]\label{lem:meas}
	Assume~\eqref{eq:S}. Then:
	\begin{enumerate}[(i)]
		\item for every $l\in\N$ and $\omega\in\Omega$,
		$\sup_{f\in\C}|\mu_l(f)(\omega)|=\sup_{g\in\C_0}|\mu_l(g)(\omega)|$; in particular $\sup_{f\in\C}|\mu_l(f)|$ is a $[0,\infty]$-valued random variable which is measurable with respect to $\sigma(Y_{lB+1},\ldots,Y_{lB+B})$;
		\item if $\hat\mu(f)=\varphi\bigl(\mu_0(f),\ldots,\mu_{K-1}(f)\bigr)$ for a continuous median rule $\varphi:\R^K\to\R$ (e.g.\ the lower median), then $\sup_{f\in\C}|\hat\mu(f)|=\sup_{g\in\C_0}|\hat\mu(g)|$ is a random variable.
	\end{enumerate}
\end{lemma}
\begin{proof}
	(i) The inequality ``$\ge$'' is trivial. Fix $f\in\C$, $\omega\in\Omega$, and let $g_m\in\C_0$ be as in~\eqref{eq:S}. Since $\mu_l(g_m)(\omega)$ is an average of $B$ values $g_m(Y_j(\omega))$ minus $\E g_m(Y_0)$, it converges to $\mu_l(f)(\omega)$. Hence $|\mu_l(f)(\omega)|=\lim_m|\mu_l(g_m)(\omega)|\le\sup_{g\in\C_0}|\mu_l(g)(\omega)|$. A countable supremum of functions measurable with respect to $\sigma(Y_{lB+1},\ldots,Y_{lB+B})$ is measurable with respect to the same $\sigma$-algebra.

	(ii) Since $\varphi$ is continuous, $\hat\mu(g_m)(\omega)\to\hat\mu(f)(\omega)$, and the argument of (i) applies. (The lower median $\varphi(x)$ of $x\in\R^K$ is nondecreasing in each coordinate and translation equivariant, $\varphi(x+t\mathbf 1)=\varphi(x)+t$; hence $|\varphi(x)-\varphi(y)|\le\max_l|x_l-y_l|$.)
\end{proof}

\begin{proof}[Proof of Theorem~\ref{thm:MoMproc}]
	\emph{Step 1 (counting).} If $m$ is a median of $x_0,\ldots,x_{K-1}$ and $|m|\ge\delta$ (resp.\ $|m|>\delta$), then $\#\{l:|x_l|\ge\delta\}\ge K/2$ (resp.\ $\#\{l:|x_l|>\delta\}\ge K/2$): if $m\ge\delta$, for at least $K/2$ indices $l$,
	$x_l\ge\delta$, and the case $m\le-\delta$ is symmetric. The bound $K/2$ cannot be improved to a strict inequality: for $K=2$, $x=(0,2\delta)$ and the midpoint median $m=\delta$, only one block deviates.

	\emph{Step 2 (the supremum).} Let $s_l:=\sup_{f\in\C}|\mu_l(f)|\in[0,\infty]$ and $Z_l:=\ind_{\{s_l\ge\delta\}}$. We claim that
	\begin{equation}\label{eq:inclusion}
		\Bigl\{\sup_{f\in\C}|\hat\mu(f)|\ge\delta\Bigr\}
		\subseteq
		\Bigl\{\sum_{l=0}^{K-1}Z_l\ge \frac K2\Bigr\}.
	\end{equation}
	Suppose that $\sum_lZ_l<K/2$. Then $L:=\{l:s_l<\delta\}$ is nonempty and $\delta':=\max_{l\in L}s_l<\delta$. For every $f\in\C$ we have $\{l:|\mu_l(f)|>\delta'\}\subseteq\{l:s_l\ge\delta\}$, a set of fewer than $K/2$ elements, so by the strict version of Step~1, $|\hat\mu(f)|\le\delta'$. Hence $\sup_f|\hat\mu(f)|\le\delta'<\delta$. Note that this argument does not require the supremum to be attained.

	\emph{Step 3 (centering).} By Lemma~\ref{lem:meas}, $Z_l$ is a random variable with $\E Z_l\le\eps_{\delta,B}$, hence $\frac12-\frac1K\sum_l\E Z_l\ge\Delta$ and
	\[
	\Bigl\{\sum_{l=0}^{K-1}Z_l\ge\frac K2\Bigr\}
	\subseteq
	\Bigl\{\frac1K\Big|\sum_{l=0}^{K-1}(Z_l-\E Z_l)\Big|\ge\Delta\Bigr\}.
	\]

	\emph{Step 4 (mixing of the indicators).} By Lemma~\ref{lem:meas}, $Z_l$ is $\sigma(Y_{lB+1},\ldots,Y_{lB+B})$-measurable. Consequently $\sigma(Z_j:j\le k)\subseteq\sigma(Y_i:i\le(k+1)B)$ and $\sigma(Z_j:j\ge k+q+1)\subseteq\sigma(Y_i:i\ge(k+q+1)B+1)$; the gap between the two index sets is $qB+1$, whence
	\[
	\alpha^Z(q+1)\le\alpha^Y(qB+1)\le\alpha^Y(qB).
	\]

	\emph{Step 5.} Lemma~\ref{lem:Rio61}, applied to $(Z_{k-1})_{k\ge1}$ with $n=K$ and $\delta=\Delta$ is applicable, because $q\le\frac27K\Delta$ is equivalent to $\Delta\ge\frac{7q}{2K}$, and $q<K$ since $\Delta\le\frac12$. Combining the previous steps with the monotonicity of $\Rb_{K,q}$ in its second argument completes the proof.
\end{proof}

\begin{remark}[Decay rates]\label{rem:decrate}
	The right-hand side of Theorem~\ref{thm:MoMproc} depends on the class $\C$ only through $\Delta$. If $\eps_{\delta,B}\le\frac14$, then $\Delta\ge\frac14$ and the choices of $q$ in Corollary~\ref{cor:specmix} apply verbatim: for fixed $\delta$ and $B$, $\P(\sup_f|\hat\mu(f)|\ge\delta)=O(n^{-a}\log^an)$ if $\alpha^Y(n)=O(n^{-a})$, and $=O(\kappa^{\sqrt n})$ if $\alpha^Y(n)=O(\rho^n)$.

	If one prefers to avoid Lemma~\ref{lem:Rio61} altogether, Chebyshev's inequality applied to $\sum_l Z_l$, together with $|\cov(Z_i,Z_j)|\le\alpha^Z(|i-j|)$ and Step~4, gives the elementary polynomial bound
	\[
	\P\Bigl(\sup_{f\in\C}|\hat\mu(f)|\ge\delta\Bigr)
	\le
	\frac{\frac14+2\sum_{k=0}^{K-1}\alpha^Y(kB+1)}{K\Delta^2}.
	\]
\end{remark}

\begin{remark}\label{rem:VC}
	Condition~\eqref{eq:S} cannot be dropped from Lemma~\ref{lem:meas}, and it is not implied
	by the Vapnik--Chervonenkis property of $\C$.
	Let $A\subset\R$ be a non-Borel set and consider the class $\C=\{\ind_{\{a\}}:a\in A\}$ of singleton indicators, whose VC dimension is~$1$. For every $y\in\R$, $\sup_{f\in\C}f(y)=\ind_A(y)$, which is not Borel measurable. 
	
	Moreover, density in $L^\infty(\P_{Y_0})$ (i.e.\ with respect to the essential supremum) is not sufficient either: if $\P_{Y_0}$ is non-atomic, every element of $\C\cup\{0\}$ equals $0$ almost surely, so $\{0\}$ is dense in $\C\cup\{0\}$ in $L^\infty(\P_{Y_0})$, yet $\sup_{f\in\C\cup\{0\}}f(Y_0)=\ind_A(Y_0)$ is not a random variable when, for instace, $\Omega=[0,1]$ with the Borel $\sigma$-algebra and the Lebesgue measure, $Y_0(\omega)=\omega$, and $A\subset[0,1]$ is not Borel. The uniform norm in condition (b) above is therefore essential. 
	
	On the other hand, \eqref{eq:S} is used in the proof of Theorem~\ref{thm:MoMproc} only to make the indicators $Z_l$ measurable: if they are replaced by $\sigma(Y_{lB+1},\ldots,Y_{lB+B})$-measurable indicator majorants $Z_l'\ge Z_l$ with $\E Z_l'\le\eps_{\delta,B}$ (e.g.\ indicators of measurable hulls), the same proof bounds the outer probability of the event in Theorem~\ref{thm:MoMproc}.
\end{remark}

\begin{corollary}[Adversarial contamination]\label{cor:contamination}
	Suppose that the observations belonging to at most $n_o$ of the $K$ blocks are replaced by arbitrary values (possibly depending on the whole sample, and not necessarily measurable), and let $\hat\mu^{o}(f)$ denote the MoM process computed from the corrupted sample. If
	\[
	\Delta_o:=\tfrac12-\tfrac{n_o}{K}-\eps_{\delta,B}>0,
	\]
	then, with $\P^*$ denoting outer probability,
	\[
	\P^*\Bigl(\sup_{f\in\C}|\hat\mu^o(f)|\ge\delta\Bigr)\le\Rb_{K,q}\bigl(\Delta_o,\alpha^Y(qB)\bigr),\qquad\text{for every integer }1<q\le\tfrac27K\Delta_o.
	\]
	In particular, $n_o$ arbitrary outliers (which affect at most $n_o$ blocks) can be tolerated as long as $n_o<K(\frac12-\eps_{\delta,B})$.
\end{corollary}
\begin{proof}
	In Step~2 of the proof of Theorem~\ref{thm:MoMproc}, apply the argument to the corrupted block means and count every corrupted block as a deviating one: the set $\{l:|\mu^o_l(f)|>\delta'\}$ is contained in $\{l:s_l\ge\delta\}\cup\{\text{corrupted blocks}\}$, where $s_l$ refers to the uncorrupted sample. This gives $\{\sup_f|\hat\mu^o(f)|\ge\delta\}\subseteq\{\sum_lZ_l\ge K/2-n_o\}$, and Steps~3--5 apply with $\Delta_o$ in place of $\Delta$.
\end{proof}

The single-block probabilities $\eps_{\delta,B}$ can be bounded by a covering argument whenever the class is Lipschitz in a parameter ranging over a totally bounded metric space. For a totally bounded metric space $(\Theta,d)$ and $\eta>0$, let $N(\Theta,\eta)$ denote the smallest number of points $\theta_1,\ldots,\theta_N\in\Theta$ such that every $\theta\in\Theta$ satisfies $d(\theta,\theta_i)\le\eta$ for some $i$.

\begin{proposition}[Lipschitz-parametrized classes]\label{prop:Lipf}
	Let $\C=\{f_\theta:\theta\in\Theta\}$, where $(\Theta,d)$ is a totally bounded metric space, $\E|f_\theta(Y_0)|<\infty$, and
	\[
	|f_\theta(y)-f_{\theta'}(y)|\le d(\theta,\theta')\,\Lambda(y),\qquad \theta,\theta'\in\Theta,\ y\in\Y,
	\]
	for a measurable $\Lambda\ge0$ with $\E\Lambda(Y_0)<\infty$. Then $\C$ satisfies~\eqref{eq:S}, and for every $l\in\N$ and $\delta,\eta>0$,
	\[
	\P\Bigl(\sup_{\theta\in\Theta}|\mu_l(f_\theta)|\ge\delta\Bigr)
	\le
	N(\Theta,\eta)\,\sup_{\theta\in\Theta}\P\bigl(|\mu_l(f_\theta)|\ge\delta/2\bigr)
	+\frac{4\eta\,\E\Lambda(Y_0)}{\delta}.
	\]
	If $\Lambda\equiv L$ is constant and $4L\eta<\delta$, the last term can be omitted.
\end{proposition}
\begin{proof}
	The metric space $\Theta$ is separable, and $\theta\mapsto f_\theta(y)$ as well as $\theta\mapsto\E f_\theta(Y_0)$ are Lipschitz, so~\eqref{eq:S} holds by (c). Let $\theta_1,\ldots,\theta_N$ be an $\eta$-net. If $d(\theta,\theta_i)\le\eta$ then $|\mu_l(f_\theta)-\mu_l(f_{\theta_i})|\le W_l:=\eta\bigl(B^{-1}\sum_{j=1}^B\Lambda(Y_{lB+j})+\E\Lambda(Y_0)\bigr)$. If $\max_i|\mu_l(f_{\theta_i})|<\delta/2$ and $W_l<\delta/2$, then $\sup_\theta|\mu_l(f_\theta)|\le\max_i|\mu_l(f_{\theta_i})|+W_l<\delta$. Hence 
	\[
	\{\sup_\theta|\mu_l(f_\theta)|\ge\delta\}\subseteq\bigcup_i\{|\mu_l(f_{\theta_i})|\ge\delta/2\}\cup\{W_l\ge\delta/2\},
	\]
	and Markov's inequality gives $\P(W_l\ge\delta/2)\le 2\E W_l/\delta=4\eta\E\Lambda(Y_0)/\delta$. If $\Lambda\equiv L$, then $W_l=2L\eta<\delta/2$.
\end{proof}

Unlike Lemma 2.2, the following lemma has explicit constants and requires no rate on the mixing coefficients.
We write
\[
A_B:=\sum_{k=1}^B\alpha^X(k)\quad\text{and}\quad
A_\infty:=\sum_{k\ge1}\alpha^X(k).
\]
\begin{lemma}[Single-block bounds]\label{lem:block}
	Let $(X_n)_{n\in\N}$ be identically distributed, $l\in\N$, $\eta>0$, and let $\mu_l(g)=\frac1B\sum_{j=1}^Bg(X_{lB+j})-\E g(X_0)$.
	\begin{enumerate}[(i)]
		\item For a measurable set $S$ and $p=\P(X_0\in S)$,
		\[
		\P\bigl(|\mu_l(\ind_S)|\ge\eta\bigr)\le\frac{p(1-p)+2\sum_{k=1}^B\min\{\alpha^X(k),\,p(1-p)\}}{B\eta^2}\le\frac{\frac14+2A_B}{B\eta^2}.
		\]
		
		\item If $|g|\le a$, then
		$\displaystyle\P\bigl(|\mu_l(g)|\ge\eta\bigr)\le\frac{a^2(1+8A_B)}{B\eta^2}$.
		
		\item If $\E|g(X_0)|<\infty$, then for every $M\ge0$,
		\[
		\P\bigl(|\mu_l(g)|\ge\eta\bigr)\le\frac{4M^2(1+8A_B)}{B\eta^2}+\frac{4\,\E\bigl(|g(X_0)|-M\bigr)_+}{\eta}.
		\]
	\end{enumerate}
\end{lemma}
\begin{proof}
	All three bounds follow from Chebyshev's inequality and the bound
	\[
	\var\Bigl(\sum_{j=1}^BW_j\Bigr)\le\sum_{i,j}|\cov(W_i,W_j)|\le B\Bigl(\gamma_0+2\sum_{k=1}^B\gamma_k\Bigr),
	\]
	valid whenever $|\cov(W_i,W_j)|\le\gamma_{|i-j|}$.
	
	(i) For two events of probability $p$, $|\P(G\cap H)-p^2|\le p(1-p)$, and $|\P(G\cap H)-\P(G)\P(H)|\le\alpha^X(k)$ if $G\in\sigma(X_i)$, $H\in\sigma(X_{i+k})$, by the very definition~\eqref{eq:mix_coef}.
	
	(ii) The covariance inequality 
	\[
	|\cov(U,V)|\le4\alpha(\sigma(U),\sigma(V))\|U\|_\infty\|V\|_\infty
	\] 
	for bounded random variables (see Theorem~A.5 in~\cite{hall2014martingale}; an $L^p$ version is due to Davydov~\cite{davydov1968convergence}) gives $\gamma_k=4a^2\alpha^X(k)$, $k\ge1$, and $\gamma_0=a^2$.
	
	(iii) Let $g_M=\max(-M,\min(M,g))$ and $r_M=g-g_M$, so that $|r_M|=(|g|-M)_+$. Then $\mu_l(g)=\mu_l(g_M)+\mu_l(r_M)$, $\P(|\mu_l(g_M)|\ge\eta/2)$ is bounded by (ii), and $\P(|\mu_l(r_M)|\ge\eta/2)\le2\E|\mu_l(r_M)|/\eta\le4\E|r_M(X_0)|/\eta$ by Markov's inequality.
\end{proof}

\begin{remark}
	If $\E|g(X_0)|^r<\infty$ for some $r>1$, then $\E(|g(X_0)|-M)_+\le\E|g(X_0)|^rM^{1-r}$, and the choice $M\asymp(B\eta)^{1/(r+1)}$ in Lemma~\ref{lem:block}(iii) gives, when $A_\infty<\infty$, a bound of order $\eta^{-2r/(r+1)}B^{-(r-1)/(r+1)}$. Both this bound and Lemma~\ref{lem:Hansen} (if $a>1$) guarantee $\eps_{\delta,B}\le\frac14$ as soon as $B\gtrsim\delta^{-2r/(r-1)}$. The sharp form of (i) is important for indicators of rare events, e.g.\ when estimating high quantiles in Section~\ref{sec:VaR}.
\end{remark}

\section{Nonparametric regression}\label{sec:regression}

Let $(X_t,Y_t)_{t\ge1}$ be an identically distributed $\R^d\times\R$-valued sequence, and let $\psi:\R\to\R$ be a measurable function with $\E|\psi(Y_1)|<\infty$. We wish to estimate the regression function
\[
m(x)=\E\bigl[\psi(Y_1)\mid X_1=x\bigr],\qquad x\in D,
\]
on a compact set $D\subset\R^d$. Kernel estimators of $m$ under mixing conditions have a long history. Roussas~\cite{ROUSSAS1990107} proved strong consistency of the Nadaraya--Watson estimator, uniformly over compacts and with rates, for strictly stationary sequences satisfying any of the four standard mixing conditions, allowing unbounded $\psi$. Bosq~\cite{bosq1993bernstein} obtained Bernstein-type large deviation inequalities for strongly mixing partial sums, which yield exponential rates under exponential moment conditions; and Thanh and Yin~\cite{THANH2015195} established sharp complete convergence results for weighted sums $\sum_iw_{ni}\xi_i$ of strongly mixing variables, with an application to fixed-design regression estimators $\sum_iw_{ni}(x)Y_{i}$. These estimators are asymptotic, or rely on (exponential) moment assumptions, and are not robust to outliers.
For independent data whose noise has only a finite second moment, Ben-Hamou and Guyader \cite{benhamou2025deviation} showed that median-of-means versions of nearest-neighbor and kernel regression estimators attain optimal deviation bounds.

Applying the median-of-means principle to the numerator and the denominator, we obtain an estimator with explicit non-asymptotic deviation bounds for $\alpha$-mixing data that are uniform over $D$, require only $\E|\psi(Y_1)|<\infty$, and allow corrupted blocks; the required block size is determined by higher moments $\E|\psi(Y_1)|^r<\infty$, $r>1$ (Corollary \ref{cor:NW}).

\medskip 
Let $\Ker:\R^d\to[0,\infty)$ be a kernel (we write $\Ker$, since $K$ denotes the number of blocks) such that
\[
\int\Ker=1,\qquad \|\Ker\|_\infty<\infty,\qquad |\Ker(u)-\Ker(v)|\le L_\Ker\|u-v\|_\infty,
\]
and put $\Ker_h(u)=h^{-d}\Ker(u/h)$ for a bandwidth $h>0$. The block statistics are
\[
\hat g_l(x)=\frac1B\sum_{j=1}^B\psi(Y_{lB+j})\Ker_h(x-X_{lB+j}),
\qquad
\hat f_l(x)=\frac1B\sum_{j=1}^B\Ker_h(x-X_{lB+j}),
\]
and the \emph{MoM Nadaraya--Watson estimator} is
\[
\hat m(x)=\frac{\hat g(x)}{\hat f(x)},\qquad
\hat g(x)=\med\bigl(\hat g_0(x),\ldots,\hat g_{K-1}(x)\bigr),\quad
\hat f(x)=\med\bigl(\hat f_0(x),\ldots,\hat f_{K-1}(x)\bigr),
\]
with the lower median, and $\hat m(x):=0$ if $\hat f(x)=0$. Its smoothed target is $m_h=g_h/f_h$, where
\[
g_h(x)=\E\bigl[\psi(Y_1)\Ker_h(x-X_1)\bigr]=\E\bigl[m(X_1)\Ker_h(x-X_1)\bigr],\qquad f_h(x)=\E\Ker_h(x-X_1).
\]
We write $N(D,\eta)$ for the covering numbers of $D$ with respect to $\|\cdot\|_\infty$; for $D=[-R,R]^d$, $N(D,\eta)\le\lceil R/\eta\rceil^d$.

\begin{theorem}\label{thm:NW}
	Assume that $(X_t,Y_t)_{t\ge1}$ is $\alpha$-mixing with coefficients $\alpha(k)$, $A_B=\sum_{k=1}^B\alpha(k)$, and that at most $n_o$ blocks are corrupted. Let $b_h:=\inf_{x\in D}f_h(x)>0$, $M_h:=\sup_{x\in D}|m_h(x)|$, $\delta_1>0$ and $0<\delta_2<b_h$. For $\eta_1>0$ and $M\ge0$, let
	\begin{align*}
		\eps_1&:=N(D,\eta_1)\Bigl[\frac{16M^2(1+8A_B)}{B\delta_1^2}+\frac{8\,\E\bigl(h^{-d}\|\Ker\|_\infty|\psi(Y_1)|-M\bigr)_+}{\delta_1}\Bigr]+\frac{4\eta_1L_\Ker\,\E|\psi(Y_1)|}{h^{d+1}\delta_1},\\
		\eps_2&:=N\Bigl(D,\frac{\delta_2h^{d+1}}{8L_\Ker}\Bigr)\,\frac{4\|\Ker\|_\infty^2(1+8A_B)}{Bh^{2d}\delta_2^2},
	\end{align*}
	and $\Delta_i:=\frac12-\frac{n_o}K-\eps_i$. If $\Delta_1,\Delta_2>0$, then for every integer $1<q\le\frac27K\min(\Delta_1,\Delta_2)$,
	\[
	\P^*\Bigl(\sup_{x\in D}\bigl|\hat m(x)-m_h(x)\bigr|>\frac{\delta_1+M_h\delta_2}{b_h-\delta_2}\Bigr)
	\le\Rb_{K,q}\bigl(\Delta_1,\alpha(qB)\bigr)+\Rb_{K,q}\bigl(\Delta_2,\alpha(qB)\bigr).
	\]
\end{theorem}
\begin{proof}
	\emph{Deterministic step.} If $|a-A|\le\delta_1$ and $|b-F|\le\delta_2<F$, then
	\[
	\Bigl|\frac ab-\frac AF\Bigr|=\Bigl|\frac{(a-A)+\frac AF(F-b)}{b}\Bigr|\le\frac{\delta_1+|A/F|\,\delta_2}{F-\delta_2}.
	\]
	Applied with $a=\hat g(x)$, $A=g_h(x)$, $b=\hat f(x)$, $F=f_h(x)\ge b_h$, it shows that the event on the left-hand side is contained in $\{\sup_x|\hat g(x)-g_h(x)|\ge\delta_1\}\cup\{\sup_x|\hat f(x)-f_h(x)|\ge\delta_2\}$. Since the median is translation equivariant, $\hat g(x)-g_h(x)$ and $\hat f(x)-f_h(x)$ are the MoM processes of the classes $\C_1=\{(u,v)\mapsto\psi(v)\Ker_h(x-u):x\in D\}$ and $\C_2=\{u\mapsto\Ker_h(x-u):x\in D\}$.

	\emph{Probabilistic step.} Both classes are Lipschitz in $x$: $|\psi(v)\Ker_h(x-u)-\psi(v)\Ker_h(x'-u)|\le\|x-x'\|_\infty L_\Ker h^{-d-1}|\psi(v)|$, and similarly with $|\psi(v)|$ replaced by $1$. By Proposition~\ref{prop:Lipf} they satisfy~\eqref{eq:S}. For $\C_1$, Proposition~\ref{prop:Lipf} with $\Lambda(u,v)=L_\Ker h^{-d-1}|\psi(v)|$ and Lemma~\ref{lem:block}(iii) at level $\delta_1/2$ (note $|\psi(v)\Ker_h(x-u)|\le h^{-d}\|\Ker\|_\infty|\psi(v)|$) show that $\eps_1$ bounds the single-block probabilities at level $\delta_1$. For $\C_2$, the choice $\eta=\delta_2h^{d+1}/(8L_\Ker)$ makes the Lipschitz term vanish in Proposition~\ref{prop:Lipf}, and Lemma~\ref{lem:block}(ii) with $a=h^{-d}\|\Ker\|_\infty$ at level $\delta_2/2$ yields $\eps_2$. Corollary~\ref{cor:contamination} applied to both classes completes the proof.
\end{proof}

\begin{lemma}[Bias]\label{lem:NWbias}
	If $\Ker(u)=0$ for $\|u\|_\infty>1$, $f_h(x)>0$, and $|m(u)-m(x)|\le\omega_m(h)$ whenever $\|u-x\|_\infty\le h$, then $|m_h(x)-m(x)|\le\omega_m(h)$.
\end{lemma}
\begin{proof} We have
	$m_h(x)=\E[w\,m(X_1)]/\E[w]$ with the weight $w=\Ker_h(x-X_1)\ge0$, which vanishes unless $\|X_1-x\|_\infty\le h$. Hence $|m_h(x)-m(x)|\le\E[w\,|m(X_1)-m(x)|]/\E w\le\omega_m(h)$.
\end{proof}

\begin{corollary}\label{cor:NW}
	Assume, in addition, that $A_\infty=\sum_k\alpha(k)<\infty$ and $\E|\psi(Y_1)|^r<\infty$ for some $r>1$. Fix $h$, $\delta_1$ and $\delta_2<b_h$. Then there is $B_0$ such that $\eps_1,\eps_2\le\frac18$ for every $B\ge B_0$ (choose $\eta_1$ small, then $M$ and $B$ large); hence $\Delta_1,\Delta_2\ge\frac14$ whenever $n_o\le K/8$, and, since $\Rb_{K,q}$ is nonincreasing in $\Delta$, the bound of Theorem~\ref{thm:NW} is at most $2\Rb_{K,q}(\frac14,\alpha(qB))$. 
	
	In particular, for fixed $h$ and $B\ge B_0$, as $K\to\infty$, the probability that $\sup_D|\hat m-m_h|$ exceeds a fixed tolerance is $O(n^{-a}\log^an)$ if $\alpha(k)=O(k^{-a})$, $a>1$, and $O(\kappa^{\sqrt n})$ if $\alpha(k)=O(\rho^k)$, where $n=KB$ (Remark~\ref{rem:decrate}). If, in addition, $\Ker$ vanishes outside $\{\|u\|_\infty\le1\}$, $X_1$ has a density that is continuous and bounded away from $0$ on a neighborhood $D'$ of $D$, and $m$ is bounded and uniformly continuous on $D'$, then $\liminf_{h\to0}b_h>0$, $\sup_{h\le h_0}M_h<\infty$ and $\sup_D|m_h-m|\to0$ by Lemma~\ref{lem:NWbias}; hence $\sup_D|\hat m-m|\to0$ in outer probability along suitable sequences $h_n,\delta_{1,n},\delta_{2,n}\to0$ and $B_n,K_n\to\infty$.
\end{corollary}

\section{Kernel density estimation}\label{sec:KDE}

Let $X_1, \ldots, X_n$ be identically distributed random variables with common probability density $f$ on $\R^d$.
The \emph{kernel density estimator} (KDE), also known as the Parzen--Rosenblatt estimator, is the nonparametric estimator defined by
\[
\widehat f_n(x) = \frac{1}{n h^d} \sum_{i=1}^n \Ker\!\left( \frac{X_i - x}{h} \right), \qquad x \in \R^d,
\]
where $h>0$ is a bandwidth parameter and $\Ker:\R^d \to [0,\infty)$ is an integrable kernel satisfying
\[
\int_{\R^d} \Ker(u)\,du = 1,
\]
see Tsybakov~\cite{tsybakov2009introduction}.
The function $\Ker$ is referred to as the \emph{kernel}, and the bandwidth $h$ acts as a smoothing parameter governing the bias--variance trade-off of $\widehat f_n$.

Kim and Lee \cite{KIM2005273} proved consistency and asymptotic normality of the KDE for possibly nonstationary $\alpha$-mixing sequences, N'drin~\cite{n2025hall} studied the recursive kernel density estimator of Hall and Patil for strongly mixing processes; Quintela-del-R\'io~\cite{Quintela-Del-Rio01072008} considered kernel estimation of the conditional hazard function for functional strongly mixing data. Uniform rates are given by
Bosq~\cite{bosq1998nonparametric}, Gy\"orfi et al.~\cite{gyorfi1989nonparametric}, and Hansen~\cite{hansen2008uniform}. None of these estimators is robust to outliers.

Humbert, Le Bars, and Minvielle~\cite{humbert2022robust} introduced the MoM-KDE in the $\mathcal O\cup\mathcal I$ framework of~\cite{lecue2019learning,lecue2020robust}: the sample consists of i.i.d.\ inliers and of an arbitrary (possibly adversarial) set of outliers $\mathcal O$. They proved that if the number $S$ of blocks satisfies $S\ge2|\mathcal O|+1$, then with probability at least $1-e^{-\gamma}$ the $L^\infty$ error is, up to a $\log S$ term, at most that of the outlier-free KDE computed from $n/S$ observations, and that the estimator is $L^1$-consistent. We now extend these results to $\alpha$-mixing data.

\paragraph{Setting.} Let $(X_t)_{t\ge1}$ be an $\R^d$-valued, identically distributed, $\alpha$-mixing sequence with coefficients $\alpha(k)$ and common density $f$, and let $A_B=\sum_{k=1}^B\alpha(k)$. In contrast with~\cite{humbert2022robust}, the blocks are formed by \emph{consecutive} observations: a random partition would destroy the dependence structure on which Lemma~\ref{lem:Rio61} relies (for i.i.d.\ data the two constructions have the same distribution). The block estimates and the MoM-KDE are
\[
\hat f_l(x)=\frac1{Bh^d}\sum_{j=1}^B\Ker\Bigl(\frac{x-X_{lB+j}}{h}\Bigr),\qquad
\hat f_{\rm MoM}(x)=\med\bigl(\hat f_0(x),\ldots,\hat f_{K-1}(x)\bigr)
\]
(lower median, so that $\hat f_{\rm MoM}$ is continuous) and, when a density is required and $\int\hat f_{\rm MoM}>0$, $\tilde f=\hat f_{\rm MoM}/\int\hat f_{\rm MoM}$. (Humbert et al.~\cite{humbert2022robust} write the kernel as $\Ker((X_i-x)/h)$; for symmetric kernels, $\Ker((x-X)/h)=\Ker((X-x)/h)$.) We assume:
\begin{enumerate}[({K}1)]
	\item $\Ker\ge0$, $\int\Ker=1$, $\|\Ker\|_\infty<\infty$, $|\Ker(u)-\Ker(v)|\le L_\Ker\|u-v\|_\infty$, and $\int\Ker(u)\|u\|^\beta du<\infty$;
	\item $f$ belongs to the H\"older class $\Sigma(\beta,L)$, $0<\beta\le1$: $|f(x)-f(x')|\le L\|x-x'\|^\beta$.
\end{enumerate}
Assumption~(K1) is satisfied by the Gaussian, Epanechnikov, triangular, and many other kernels. Under~(K1)--(K2) the bias satisfies
\begin{equation}\label{eq:KDEbias}
	\bigl|\E\hat f_l(x)-f(x)\bigr|=\Bigl|\int\Ker(u)\bigl(f(x-hu)-f(x)\bigr)du\Bigr|\le C_2h^\beta,\qquad C_2:=L\int\Ker(u)\|u\|^\beta du,
\end{equation}
by the change of variables $y=x-hu$ in $\E\Ker_h(x-X_1)=\int\Ker_h(x-y)f(y)\,dy$.

\begin{theorem}[$L^\infty$ error of the MoM-KDE under $\alpha$-mixing]\label{thm:KDE}
	Assume (K1)--(K2), let $D\subset\R^d$ be compact, $\delta>0$, $h>0$, and suppose that at most $n_o$ blocks are corrupted. Let
	\[
	\eps_{\delta,B,h}:=N\Bigl(D,\frac{\delta h^{d+1}}{8L_\Ker}\Bigr)\,\frac{4\|\Ker\|_\infty^2\,(1+8A_B)}{B\,h^{2d}\,\delta^2},
	\qquad
	\Delta:=\frac12-\frac{n_o}{K}-\eps_{\delta,B,h}.
	\]
	If $\Delta>0$, then for every integer $1<q\le\frac27K\Delta$,
	\[
	\P^*\Bigl(\sup_{x\in D}\bigl|\hat f_{\rm MoM}(x)-f(x)\bigr|>\delta+C_2h^\beta\Bigr)\le\Rb_{K,q}\bigl(\Delta,\alpha(qB)\bigr).
	\]
\end{theorem}
\begin{proof}
	By~\eqref{eq:KDEbias} and the translation equivariance of the median, the event is contained in
	\[
	\Bigl\{\sup_{x\in D}\bigl|\hat f_{\rm MoM}(x)-\E\hat f_0(x)\bigr|>\delta\Bigr\},
	\]
	i.e.\ in the event that the MoM process of the class $\{y\mapsto\Ker_h(x-y):x\in D\}$ exceeds $\delta$. This class is bounded by $a=h^{-d}\|\Ker\|_\infty$ and Lipschitz in $x$ with the constant $L_\Ker h^{-d-1}$; for $\eta=\delta h^{d+1}/(8L_\Ker)$ the Lipschitz term of Proposition~\ref{prop:Lipf} vanishes, and Lemma~\ref{lem:block}(ii) at level $\delta/2$ gives the single-block bound $\eps_{\delta,B,h}$. Corollary~\ref{cor:contamination} completes the proof.
\end{proof}

\begin{corollary}\label{cor:KDE}
	Assume, in addition, that $\alpha(k)\le c_0\rho^k$ for some $\rho\in(0,1)$, and let $D=[-R,R]^d$. For fixed $\delta,h$ choose
	\[
	B\ge B(\delta,h):=\frac{16\|\Ker\|_\infty^2(1+8A_\infty)}{h^{2d}\delta^2}\Bigl\lceil\frac{8L_\Ker R}{\delta h^{d+1}}\Bigr\rceil^d,
	\]
	so that $\eps_{\delta,B,h}\le\frac14$. If $n_o\le K/8$, then $\Delta\ge\frac18$ and, with $q=\lfloor\sqrt{K/B}\rfloor$,
	\[
	\P^*\Bigl(\sup_{x\in D}|\hat f_{\rm MoM}(x)-f(x)|>\delta+C_2h^\beta\Bigr)=O\bigl(\kappa^{\sqrt n}\bigr),\qquad n=KB,\ K\to\infty,\ B\ge B(\delta,h)\text{ fixed},
	\]
	for some $\kappa\in(0,1)$. Thus, as in~\cite{humbert2022robust}, the error bound depends on the outliers only through their number, and not on their type.
\end{corollary}

\begin{remark}
	The block size required in Corollary~\ref{cor:KDE} is far from optimal, because Lemma~\ref{lem:block}(ii) only uses $\|\Ker_h\|_\infty$. If the joint densities of $(X_t,X_{t+k})$, $t,k\ge1$, are bounded by $M_2$ and $\|f\|_\infty<\infty$, then $\var\Ker_h(x-X_1)\le\|\Ker\|_\infty\|f\|_\infty h^{-d}$ and 
	\[
	|\cov(\Ker_h(x-X_1),\Ker_h(x-X_{1+k}))|\le\min\{M_2+\|f\|_\infty^2,\,4\alpha(k)\|\Ker\|_\infty^2h^{-2d}\}; 
	\]
	splitting the sum of covariances at $k\asymp h^{-d}$ (Masry's technique, see~\cite{bosq1998nonparametric}) shows that $\var\hat f_l(x)=O((Bh^d)^{-1})$ under $\alpha(k)=O(k^{-a})$ with $a>2$, and the factor $h^{-2d}$ in $\eps_{\delta,B,h}$ improves to $h^{-d}$. 
	
	Since Theorem~\ref{thm:MoMproc} takes the single-block bound as an input, any sharper single-block bound, e.g.\ one based on the Bernstein inequality for strongly mixing sequences of Merlev\`ede, Peligrad and Rio~\cite{merlevede2009bernstein}, can be inserted directly.
\end{remark}

\begin{proposition}[$L^1$ consistency]\label{prop:KDEL1}
	Let $\mathcal S$ be the set of uncorrupted blocks and assume $K>2n_o$, i.e.\ $|\mathcal S|>K/2$.
	\begin{enumerate}[(i)]
		\item Deterministically, $\|\hat f_{\rm MoM}-f\|_1\le\sum_{l\in\mathcal S}\|\hat f_l-f\|_1$ and, if $\|\hat f_{\rm MoM}-f\|_1<1$ (which implies $\int\hat f_{\rm MoM}>0$), $\|\tilde f-f\|_1\le2\|\hat f_{\rm MoM}-f\|_1$.
		\item If $K$ is fixed, $A_\infty<\infty$, $h\to0$ and $Bh^{2d}\to\infty$, then $\|\tilde f-f\|_1\to0$ in outer probability.
	\end{enumerate}
\end{proposition}
\begin{proof}
	(i) Since the uncorrupted blocks form a strict majority, every median lies between 
	the minimum and the maximum of $\hat f_l(x)$ over $l\in\mathcal S$
	(otherwise more than $K/2$ values would lie strictly on one side of it). Hence $|\hat f_{\rm MoM}(x)-f(x)|\le\max_{l\in\mathcal S}|\hat f_l(x)-f(x)|\le\sum_{l\in\mathcal S}|\hat f_l(x)-f(x)|$; integrate. For the normalization, with $c=\int\hat f_{\rm MoM}>0$, $\int|\hat f_{\rm MoM}/c-f|\le|1/c-1|\,c+\int|\hat f_{\rm MoM}-f|=|1-c|+\|\hat f_{\rm MoM}-f\|_1$ and $|1-c|=|\int(f-\hat f_{\rm MoM})|\le\|\hat f_{\rm MoM}-f\|_1$.

	(ii) By (i) it suffices to show $\E\|\hat f_l-f\|_1\to0$ for each uncorrupted $l$. Write $f_h=\E\hat f_l$. Then $\|f_h-f\|_1\to0$ as $h\to0$ ($\Ker_h$ is an approximate identity), and for $Q_R=[-R,R]^d$, by the Cauchy--Schwarz inequality and the variance bound $\var\hat f_l(x)\le h^{-2d}\|\Ker\|_\infty^2(1+8A_\infty)/B$ from the proof of Lemma~\ref{lem:block}(ii),
	\[
	\E\|\hat f_l-f_h\|_1\le(2R)^dh^{-d}\|\Ker\|_\infty\Bigl(\frac{1+8A_\infty}{B}\Bigr)^{1/2}+2\int_{\R^d\setminus Q_R}f_h.
	\]
	The first term tends to $0$ for every fixed $R$, and the second one is small uniformly in $h\le1$ for large $R$.
\end{proof}

\section{Robust estimation of VaR and CVaR}\label{sec:VaR}

Let $(X_t)_{t\ge1}$ be the losses of a portfolio (e.g.\ negative log-returns) observed over time. We assume that the $X_t$ are identically distributed with distribution function $F$, and that the sequence is $\alpha$-mixing with coefficients $\alpha^X(k)$; we write $A_B=\sum_{k=1}^B\alpha^X(k)$. Many standard models of financial returns, such as GARCH and stochastic volatility models, are geometrically $\beta$-mixing, and hence $\alpha$-mixing, under mild conditions~\cite{carrasco2002mixing}; at the same time their marginal distributions are typically heavy tailed, so that only moments of order $r$ below the tail index exist. For a confidence level $\alpha\in(0,1)$ (typically 0.95 or 0.99; not to be confused with the mixing coefficients $\alpha^X(k)$), the VaR and the CVaR are
\[
\VaR_\alpha=q_\alpha=\inf\{t\in\R:F(t)\ge\alpha\},\qquad
\CVaR_\alpha=\frac1{1-\alpha}\int_\alpha^1q_u\,du=\min_{c\in\R}G(c),
\]
where $G(c)=c+\E(X_1-c)_+/(1-\alpha)$; the variational formula is due to Rockafellar and Uryasev~\cite{rockafellar2000optimization,rockafellar2002conditional}, see also Acerbi and Tasche~\cite{acerbi2002coherence}. Nonparametric estimation of both quantities from dependent returns is classical, see e.g.~\cite{chen2005nonparametric,chen2008nonparametric}; most of this theory is asymptotic (consistency, asymptotic normality, mean squared error).

Recently, Mulumudi et al.~\cite{mulumudi2026generalization} obtained minimax-optimal high-probability bounds for CVaR-based learning with heavy-tailed losses, including $\beta$-mixing data with stretched-exponentially decaying coefficients, and a truncated median-of-means estimator that is robust to oblivious adversarial contamination. The results below differ in three respects: dependence is measured by $\alpha$-mixing, which is weaker than $\beta$-mixing and we allow polynomially decaying coefficients; contamination is handled together with this dependence, and the corrupted blocks may be chosen with knowledge of the whole sample; and, for a fixed tolerance, the probability of a larger error decays faster than any power of $n$ under geometric mixing, whereas the dependent-data bound of~\cite{mulumudi2026generalization} holds with probability $1-\delta-n^{-O(1)}$. On the other hand, the accuracy rates obtained here are not optimal.

In risk management one needs bounds of the form 
$\P\bigl(|\widehat{\VaR}_\alpha-\VaR_\alpha|>\eps\bigr)\le p,$
for a prescribed tolerance $\eps$, with explicit $p=p(n,\eps)$, rather than bounds on the mean squared error. The estimators below satisfy such bounds also when the observations in a limited number of blocks are corrupted, e.g.\ by recording errors.
The median-of-means principle delivers exactly such bounds, and in addition protects against corrupted observations (data errors, flash crashes) affecting a limited number of blocks. Throughout this section, at most $n_o$ blocks may be corrupted, and $\P^*$ denotes outer probability.

\subsection{The empirical approach}

Let $F_l(t)=\frac1B\sum_{j=1}^B\ind_{\{X_{lB+j}\le t\}}$ be the empirical distribution function of the $l$-th block, and define the \emph{MoM distribution function} and the \emph{MoM Value-at-Risk}
\[
\hat F(t)=\med\bigl(F_0(t),\ldots,F_{K-1}(t)\bigr),\qquad
\widehat{\VaR}_\alpha=\inf\{t:\hat F(t)\ge\alpha\},
\]
with the lower median. Since the lower median is nondecreasing and continuous in each argument, $\hat F$ is a distribution function: it is nondecreasing and right-continuous, vanishes to the left of the smallest and equals one to the right of the largest observation. Hence the MoM Value-at-Risk is well defined.

\begin{theorem}[MoM Value-at-Risk]\label{thm:VaR}
	Let $\eps>0$ and assume that $F(q_\alpha+\eps)>\alpha$. Put $p_\pm=F(q_\alpha\pm\eps)$, $\eta_+=p_+-\alpha$, $\eta_-=\alpha-p_->0$,
	\[
	\eps_\pm:=\frac{p_\pm(1-p_\pm)+2\sum_{k=1}^B\min\{\alpha^X(k),\,p_\pm(1-p_\pm)\}}{B\eta_\pm^2},\qquad
	\Delta_\pm:=\frac12-\frac{n_o}K-\eps_\pm.
	\]
	If $\Delta_+,\Delta_->0$, then for every integer $1<q\le\frac27K\min(\Delta_+,\Delta_-)$,
	\[
	\P^*\bigl(|\widehat{\VaR}_\alpha-\VaR_\alpha|>\eps\bigr)\le\Rb_{K,q}\bigl(\Delta_+,\alpha^X(qB)\bigr)+\Rb_{K,q}\bigl(\Delta_-,\alpha^X(qB)\bigr).
	\]
	No moment assumption is needed. Without Lemma~\ref{lem:Rio61}, the right-hand side can be replaced by $\sum_\pm\bigl(\frac14+2\sum_{k=0}^{K-1}\alpha^X(kB+1)\bigr)/(K\Delta_\pm^2)$ (Remark~\ref{rem:decrate}, applied to the event $\{\sum_lZ_l\ge\frac K2-n_o\}$).
\end{theorem}
\begin{proof}
	\emph{Quantile inversion.} If $\widehat{\VaR}_\alpha>q_\alpha+\eps$, then $\hat F(q_\alpha+\eps)<\alpha$, since otherwise $q_\alpha+\eps$ would belong to the set whose infimum defines $\widehat{\VaR}_\alpha$. If $\widehat{\VaR}_\alpha<q_\alpha-\eps$, there is $t<q_\alpha-\eps$ with $\hat F(t)\ge\alpha$, and by monotonicity $\hat F(q_\alpha-\eps)\ge\alpha$. Therefore
	\[
	\{|\widehat{\VaR}_\alpha-q_\alpha|>\eps\}\subseteq\{|\hat F(q_\alpha+\eps)-F(q_\alpha+\eps)|\ge\eta_+\}\cup\{|\hat F(q_\alpha-\eps)-F(q_\alpha-\eps)|\ge\eta_-\}.
	\]
	\emph{Deviation of $\hat F$.} The difference $\hat F(t)-F(t)$ is the MoM of the block means $\mu_l(\ind_{(-\infty,t]})$, so each event is controlled by Theorem~\ref{thm:MoMproc} and Corollary~\ref{cor:contamination} with a one-element class, the single-block bound being given by Lemma~\ref{lem:block}(i). The polynomial version follows from Remark~\ref{rem:decrate}.
\end{proof}

\begin{remark}[Choice of the block size]\label{rem:VaRblock}
	If $F$ has a positive density $F'$ near $q_\alpha$, then $\eta_\pm\approx F'(q_\alpha)\eps$ and $p_\pm(1-p_\pm)\approx\alpha(1-\alpha)$. For geometrically mixing sequences $\sum_k\min\{\alpha^X(k),v\}=O(v\log(1/v))$, hence $\eps_\pm\le\frac14$ as soon as
	\[
	B\gtrsim\frac{\alpha(1-\alpha)\log\frac1{\alpha(1-\alpha)}}{F'(q_\alpha)^2\,\eps^2},
	\]
	which, up to the logarithmic factor, is the number of observations needed to resolve the quantile with accuracy $\eps$ (compare with the asymptotic variance $\alpha(1-\alpha)/F'(q_\alpha)^2$ of the empirical quantile, see Bahadur~\cite{bahadur1966note}). For $B$ with $\eps_\pm\le\frac14$, the deviation bound decays as $O(\kappa^{\sqrt n})$ under geometric mixing (Remark~\ref{rem:decrate}). For indicators no moment problem arises, and Lemma~\ref{lem:Rio61} applied to the full sample gives bounds of the same type for the empirical quantile of uncontaminated data; the benefit of the MoM construction here is its robustness against corrupted blocks and its compatibility with the CVaR estimators below, for which the median-of-means principle is essential.
\end{remark}

For the Conditional Value-at-Risk, the relevant functions $x\mapsto(x-c)_+$ are unbounded. For heavy-tailed losses the empirical average of $(X_t-c)_+$ only admits polynomial deviation bounds (a single extreme loss moves it by a fixed amount with polynomially small probability), while the median of block means does not suffer from this effect. We use the MoM version of the Rockafellar--Uryasev objective,
\[
\hat G(c)=c+\frac{1}{1-\alpha}\,\med_l\Bigl(\frac1B\sum_{j=1}^B(X_{lB+j}-c)_+\Bigr),
\]
with the lower median, so that $\hat G$ is continuous.

\begin{lemma}[Rockafellar--Uryasev]\label{lem:RU}
	Assume $\E|X_1|<\infty$ and let $\xi$ be any $\alpha$-quantile, i.e.\ $\P(X_1<\xi)\le\alpha\le\P(X_1\le \xi)$. Then for every $c\in\R$,
	\[
	0\le G(c)-G(\xi)\le\frac{|c-\xi|\,|F(c)-\alpha|}{1-\alpha}.
	\]
	In particular, $\min_cG(c)=G(q_\alpha)=\CVaR_\alpha$, and if $|F(c)-\alpha|\le L|c-q_\alpha|$ near $q_\alpha$, then $G(c)-G(q_\alpha)\le L(c-q_\alpha)^2/(1-\alpha)$.
\end{lemma}
\begin{proof}
	The function $c\mapsto(x-c)_+$ is convex, with subgradients $-\ind_{\{x>\xi\}}$ and $-\ind_{\{x\ge \xi\}}$ at $\xi$ and $-\ind_{\{x>c\}}$ at $c$. Hence 
	\begin{equation*}
		(x-c)_+\ge(x-\xi)_+-(c-\xi)\ind_{\{x>\xi\}}
		\quad\text{and}\quad
		(x-c)_+\ge(x-\xi)_++(\xi-c)\ind_{\{x\ge \xi\}},
	\end{equation*}
	which after integration give 
	\[
	G(c)-G(\xi)\ge(c-\xi)(\P(X_1\le \xi)-\alpha)/(1-\alpha)\ge0
	\] 
	for $c\ge \xi$ and 
	\[
	G(c)-G(\xi)\ge(\xi-c)(\alpha-\P(X_1<\xi))/(1-\alpha)\ge0
	\] for $c\le \xi$. Similarly, $(x-c)_+\le(x-\xi)_+-(c-\xi)\ind_{\{x>c\}}$ yields 
	\[G(c)-G(\xi)\le(c-\xi)(F(c)-\alpha)/(1-\alpha).\] 
	The identification of $G(q_\alpha)$ with $\frac1{1-\alpha}\int_\alpha^1q_u\,du$ is classical~\cite{rockafellar2002conditional,acerbi2002coherence}.
\end{proof}

\begin{theorem}[Empirical MoM-CVaR]\label{thm:CVaR}
	Assume $\E|X_1|<\infty$, let $a\le q_\alpha\le b$, and set $\widehat{\CVaR}_\alpha:=\min_{c\in[a,b]}\hat G(c)$. Let $\delta>0$ and $M\ge0$, let $N>4(b-a)/\delta$ be an integer, and define
	\[
	\eps^{\rm CV}_{\delta,B}:=(N+1)\Bigl[\frac{16M^2(1+8A_B)}{B\delta^2}+\frac{8\,\E\bigl(|X_1|+\max(|a|,|b|)-M\bigr)_+}{\delta}\Bigr],
	\qquad\Delta:=\frac12-\frac{n_o}K-\eps^{\rm CV}_{\delta,B}.
	\]
	If $\Delta>0$, then for every integer $1<q\le\frac27K\Delta$,
	\[
	\P^*\Bigl(\bigl|\widehat{\CVaR}_\alpha-\CVaR_\alpha\bigr|\ge\frac{\delta}{1-\alpha}\Bigr)\le\Rb_{K,q}\bigl(\Delta,\alpha^X(qB)\bigr).
	\]
\end{theorem}
\begin{proof}
	By Lemma~\ref{lem:RU}, $\CVaR_\alpha=\min_{[a,b]}G$. For functions with $|\hat G-G|\le s$ on $[a,b]$ the minima differ by at most $s$, and $\hat G(c)-G(c)$ is $(1-\alpha)^{-1}$ times the MoM process of the class $\{x\mapsto(x-c)_+:c\in[a,b]\}$. This class is Lipschitz in $c$ with $\Lambda\equiv1$, hence satisfies~\eqref{eq:S}, and $|(x-c)_+|\le|x|+\max(|a|,|b|)$. With the $\frac{b-a}N$-net $\{a+k(b-a)/N\}_{k=0}^N$ the Lipschitz term of Proposition~\ref{prop:Lipf} vanishes (as $4(b-a)/N<\delta$), and Lemma~\ref{lem:block}(iii) at level $\delta/2$ gives $\eps^{\rm CV}_{\delta,B}$. Corollary~\ref{cor:contamination} completes the proof.
\end{proof}

If $\E|X_1|^r<\infty$ for some $r>1$, then, with $C:=\max(|a|,|b|)$, $\E(|X_1|+C-M)_+\le\E(|X_1|+C)^rM^{1-r}$, and $\eps^{\rm CV}_{\delta,B}\le\frac14$ for $B\gtrsim(N\delta^{-2r/(r+1)})^{(r+1)/(r-1)}$: the median of block means produces deviation bounds of order $\kappa^{\sqrt n}$ under geometric mixing (Remark~\ref{rem:decrate}), although the losses have only $r$ moments.

\subsection{A Bahadur-type approach}

The classical Bahadur representation~\cite{bahadur1966note} linearizes the sample quantile around $q_\alpha$. Its analogue for CVaR is the fact that $G'(q_\alpha)=0$: plugging a a consistent estimate of the VaR into the Rockafellar--Uryasev objective does not affect the first-order error. This suggests the plug-in estimator
\[
\widetilde{\CVaR}_\alpha:=\hat G\bigl(\widehat{\VaR}_\alpha\bigr),
\]
which, unlike the estimator of Theorem~\ref{thm:CVaR}, does not require an a priori interval $[a,b]$ containing $q_\alpha$.

\begin{theorem}[Bahadur-type MoM-CVaR]\label{thm:Bahadur}
	Assume $\E|X_1|<\infty$, $F(q_\alpha)=\alpha$, and $|F(c)-\alpha|\le L|c-q_\alpha|$ for $|c-q_\alpha|\le\eps$. Let $\gamma>0$ and let $\eps^{\rm CV}_{(1-\alpha)\gamma,B}$ be defined as in Theorem~\ref{thm:CVaR} with $[a,b]=[q_\alpha-\eps,q_\alpha+\eps]$ and $\delta=(1-\alpha)\gamma$, and assume that $\Delta_\gamma=\frac12-\frac{n_o}K-\eps^{\rm CV}_{(1-\alpha)\gamma,B}>0$. Then for every integer $1<q\le\frac27K\Delta_\gamma$,
	\[
	\P^*\Bigl(\bigl|\widetilde{\CVaR}_\alpha-\CVaR_\alpha\bigr|\ge\gamma+\frac{L\eps^2}{1-\alpha}\Bigr)
	\le\P^*\bigl(|\widehat{\VaR}_\alpha-q_\alpha|>\eps\bigr)+\Rb_{K,q}\bigl(\Delta_\gamma,\alpha^X(qB)\bigr),
	\]
	and the first term on the right-hand side is bounded by Theorem~\ref{thm:VaR} if $F(q_\alpha+\eps)>\alpha$. The statement holds for any pilot estimator in place of $\widehat{\VaR}_\alpha$.
\end{theorem}
\begin{proof}
	On the event $\{|\widehat{\VaR}_\alpha-q_\alpha|\le\eps\}\cap\{\sup_{|c-q_\alpha|\le\eps}|\hat G(c)-G(c)|<\gamma\}$, Lemma~\ref{lem:RU} gives
	\[
	|\hat G(\widehat{\VaR}_\alpha)-G(q_\alpha)|\le|\hat G(\widehat{\VaR}_\alpha)-G(\widehat{\VaR}_\alpha)|+G(\widehat{\VaR}_\alpha)-G(q_\alpha)<\gamma+\frac{L\eps^2}{1-\alpha}.
	\]
	The probability of the complement of the second event is bounded as in the proof of Theorem~\ref{thm:CVaR}.
\end{proof}

Thus the error of the VaR pilot enters the deterministic part of the bound only quadratically: an error of order $\eps$ in $\widehat{\VaR}_\alpha$ costs $O(\eps^2)$ in $\widetilde{\CVaR}_\alpha$, and $\widetilde{\CVaR}_\alpha$ behaves, to first order, like the oracle estimator $\hat G(q_\alpha)$.

\begin{remark}[Choice of the parameters]
	Given a tolerance $\eps$ (or $\delta$) and an admissible risk $p$: (1) choose $B$ such that the single-block bound $\eps_\pm$ (resp.\ $\eps^{\rm CV}$) is at most $\frac14$, using a conservative upper bound for the mixing coefficients of the return model; (2) with $K=n/B$ blocks and $n_o\le K/8$ suspected corrupted blocks, $\Delta\ge\frac18$; (3) choose $q$ as in Corollary~\ref{cor:specmix} and read off $p$ from $\Rb_{K,q}(\Delta,\alpha^X(qB))$; conversely, the smallest $n$ achieving a prescribed $p$ is obtained by solving for $K$.
\end{remark}

\section{Sampling of bulk materials: reduction to one dimension}\label{sec:spatial}

The quality of bulk particulate materials, such as cement, coal, mineral ores or grain, is assessed by analyzing samples taken from a large lot, e.g.\ a heap or a stockpile. The measured property (typically the concentration of an analyte) varies throughout the lot: measurements taken at nearby points are dependent, and the dependence fades with distance. The distribution of the measurements may be heavy tailed, e.g.\ because of rare grains of a valuable or harmful constituent, and some samples may be grossly wrong because of foreign inclusions, contamination during extraction or faulty measurements. Finally, the guarantee needed in practice is of the form 
\[
\P(|\hat\mu-\mu|\ge\delta)\le p
\] 
for a prescribed tolerance $\delta$ and risk $p$, e.g.\ when a delivered lot is accepted or rejected on the basis of an estimate $\hat\mu$ of the mean grade $\mu$ of the material.
For an overview of the theory of sampling particulate materials, see~\cite{gy1998sampling,pitard2019theory,minkkinen2004practical}.

In this section we show that, although the measurements form a random field rather than a sequence, the median-of-means estimator of $\mu$ can be analyzed by the one-dimensional results of Section~\ref{sec:main}, provided the samples are taken in well-separated groups.

\paragraph{Setting.}
Let $(\mathbb H,d)$ be a metric space describing the lot (e.g.\ a bounded subset of $\R^3$), and let $C(x)$, $x\in\mathbb H$, be a real-valued random field, $C(x)$ being the result of a measurement taken at the point $x$. We assume that the variables $C(x)$, $x\in\mathbb H$, are identically distributed (not necessarily independent) and integrable, and we wish to estimate $\mu:=\E C(x)$. Identical distribution excludes systematic trends such as segregation; the results below use it only through the common mean $\mu$ and the single-group bounds. For $D\ge0$ and $N\in\N$ we define the mixing coefficients of the field by
\begin{equation}\label{eq:spatial_mixing}
	\alpha^C_N(D):=\sup_{A,A'}\alpha\bigl(\sigma(C(x):x\in A),\,\sigma(C(x):x\in A')\bigr),
\end{equation}
where the supremum is taken over all finite sets $A,A'\subseteq\mathbb H$ with $|A|+|A'|\le N$ and $d(a,a')\ge D$ for all $a\in A$, $a'\in A'$. Clearly, $\alpha^C_N(D)$ is nonincreasing in $D$ and nondecreasing in $N$.

A total of $n=KB$ increments are taken at the points $x_{l,j}\in\mathbb H$, $l=0,\ldots,K-1$, $j=1,\ldots,B$. The $l$-th \emph{group} (portions of material collected by one operation of the sampling device) consists of the $B$ increments $x_{l,1},\ldots,x_{l,B}$, taken in a small neighborhood of a chosen location, and increments belonging to different groups are at least $D$ apart:
\begin{equation}\label{eq:spatial_gap}
	d(x_{l,j},x_{l',j'})\ge D\qquad\text{whenever } l\ne l'.
\end{equation}
The estimator is the median of the group averages,
\[
\hat\mu=\med(\mu_0,\ldots,\mu_{K-1}),\qquad \mu_l=\frac1B\sum_{j=1}^BC(x_{l,j}),
\]
where, as in Corollary~\ref{cor:contamination}, the measurements of at most $n_o$ groups may be replaced by arbitrary values.

\paragraph{Serialization.}
We list the measurements group by group:
\begin{equation}\label{eq:serial}
	Y_{lB+j}:=C(x_{l,j}),\qquad l=0,\ldots,K-1,\ j=1,\ldots,B.
\end{equation}
Then $\mu_l=\frac1B\sum_{j=1}^BY_{lB+j}$ are exactly the block averages~\eqref{eq:movAvg} of the sequence $Y_1,\ldots,Y_n$, and $\hat\mu$ is its median-of-means estimator. The order of the increments within a group is irrelevant, and so is the order of the groups. For the finite sequence $Y_1,\ldots,Y_n$, the mixing coefficients~\eqref{eq:mix_coef} are computed with the $\sigma$-algebras $\sigma(Y_i:i\le k)$ and $\sigma(Y_i:k+m\le i\le n)$, $1\le k<n$.

\begin{proposition}[One-dimensional reduction]\label{prop:serial}
	Under~\eqref{eq:spatial_gap}, the sequence~\eqref{eq:serial} satisfies
	\[
	\alpha^Y(m)\le\alpha^C_n(D)\qquad\text{for every } m\ge B,
	\]
	and $\alpha\bigl(\sigma(Y_{lB+1},\ldots,Y_{lB+B}),\,\sigma(Y_{l'B+1},\ldots,Y_{l'B+B})\bigr)\le\alpha^C_{2B}(D)$ for $l\ne l'$.
\end{proposition}
\begin{proof}
	Let $1\le k<n$ and $m\ge B$. If $i\le k$ and $i'\ge k+m$, then $i'-i\ge B$, so $Y_i$ and $Y_{i'}$ belong to different groups, and the corresponding points are at distance at least $D$ by~\eqref{eq:spatial_gap}. Hence the sets $A=\{x_{l,j}:lB+j\le k\}$ and $A'=\{x_{l,j}:k+m\le lB+j\le n\}$ satisfy $d(a,a')\ge D$ for all $a\in A$, $a'\in A'$, and $|A|+|A'|\le n$. Since $\sigma(Y_i:i\le k)=\sigma(C(x):x\in A)$ and $\sigma(Y_i:k+m\le i\le n)=\sigma(C(x):x\in A')$, the first claim follows from~\eqref{eq:spatial_mixing}. The second claim is proved in the same way, with $A$ and $A'$ the sets of increments of the groups $l$ and $l'$.
\end{proof}

Thus the measurements, listed group by group, form a sequence whose mixing coefficients at lags $m\ge B$ are controlled by the spatial dependence at the distance $D$. The proofs of Theorem 2.6 and Corollary 2.10 use the mixing coefficients only through $\alpha^Y(qB)$ with $q\ge2$; the single-block bounds involve the dependence within a group (Remark~\ref{rem:spatial_block}). 

\begin{corollary}[Median-of-means for bulk sampling]\label{cor:spatial}
	Assume~\eqref{eq:spatial_gap}, let $\delta>0$ and $\eps\ge\max_l\P(|\mu_l-\mu|\ge\delta)$, where $\mu_l$ refers to the uncorrupted measurements, and suppose that $\Delta:=\frac12-\frac{n_o}K-\eps>0$. Then, with $\P^*$ denoting outer probability:
	\begin{enumerate}[(i)]
		\item for every integer $1<q\le\frac27K\Delta$,
		\[
		\P^*\bigl(|\hat\mu-\mu|\ge\delta\bigr)\le\Rb_{K,q}\bigl(\Delta,\alpha^C_n(D)\bigr);
		\]
		in particular, if $K\ge7/\Delta$, then
		\[
		\P^*\bigl(|\hat\mu-\mu|\ge\delta\bigr)\le4\Bigl(1+\frac\Delta4\Bigr)^{-\Delta K/14}+\frac{28}{\Delta}\,\alpha^C_n(D);
		\]
		\item without Lemma~\ref{lem:Rio61},
		\[
		\P^*\bigl(|\hat\mu-\mu|\ge\delta\bigr)\le\frac{\frac14+(K-1)\,\alpha^C_{2B}(D)}{K\Delta^2}\le\frac1{4K\Delta^2}+\frac{\alpha^C_{2B}(D)}{\Delta^2}.
		\]
	\end{enumerate}
\end{corollary}
\begin{proof}
	(i) The proofs of Theorem~\ref{thm:MoMproc} and Corollary~\ref{cor:contamination} only involve $Y_1,\ldots,Y_n$, so they apply to the sequence~\eqref{eq:serial} and the one-element class $\{y\mapsto y\}$. By Proposition~\ref{prop:serial}, $\alpha^Y(qB)\le\alpha^C_n(D)$, which gives the first bound. Since $\alpha^C_n(D)$ does not depend on $q$, we may take $q=2$, which is admissible if $K\ge7/\Delta$; then $K\ge14$, hence $\frac{(2/7)\Delta K}{K+2}\ge\frac27\cdot\frac78\,\Delta=\frac\Delta4$, and the second bound follows.

	(ii) Let $Z_l=\ind_{\{|\mu_l-\mu|\ge\delta\}}$, computed from the uncorrupted measurements. As in Step~1 of the proof of Theorem~\ref{thm:MoMproc} and in Corollary~\ref{cor:contamination}, $\{|\hat\mu-\mu|\ge\delta\}\subseteq\{\sum_lZ_l\ge\frac K2-n_o\}\subseteq\{\sum_l(Z_l-\E Z_l)\ge K\Delta\}$. Moreover, $\var Z_l\le\frac14$ and, by Proposition~\ref{prop:serial}, $|\cov(Z_l,Z_{l'})|\le\alpha^C_{2B}(D)$ for $l\ne l'$. Hence $\var\sum_lZ_l\le\frac K4+K(K-1)\alpha^C_{2B}(D)$, and Chebyshev's inequality completes the proof.
\end{proof}

Bound (ii) involves only the dependence between pairs of groups, whereas for random fields indexed by $\Z^d$, $d\ge2$, mixing conditions for sets of unbounded size are very restrictive..

\begin{remark}[Single-group bounds]\label{rem:spatial_block}
	The proof of Lemma~\ref{lem:block}(i)--(iii) bounds the variance of a block sum by the sum of the covariances of pairs of observations, which are controlled by the dependence between two single observations. Therefore it carries over to the group $l$ if $A_B$ is replaced by
	\[
	A^{\rm loc}_l:=\frac1{2B}\sum_{i\ne j}\alpha^C_2\bigl(d(x_{l,i},x_{l,j})\bigr),
	\]
	and, in (i), $\sum_{k=1}^B\min\{\alpha^X(k),p(1-p)\}$ by $\frac1{2B}\sum_{i\ne j}\min\{\alpha^C_2(d(x_{l,i},x_{l,j})),p(1-p)\}$. For instance, by (iii), for every $M\ge0$,
	\[
	\P\bigl(|\mu_l-\mu|\ge\delta\bigr)\le\frac{4M^2(1+8A^{\rm loc}_l)}{B\delta^2}+\frac{4\,\E\bigl(|C(x)|-M\bigr)_+}{\delta}.
	\]
	If the increments of a group are close to each other, they are strongly dependent, $A^{\rm loc}_l$ is of order $B$, and this bound does not improve as $B$ grows: averaging within a group only removes the small-scale variability of the material. A positive $\Delta$ then requires the tolerance $\delta$ to exceed the typical deviation of a group average from $\mu$, and the error probability is driven by the number $K$ of well-separated groups.
\end{remark}

\begin{remark}[Choice of the design]\label{rem:spatial_design}
	The two terms in Corollary~\ref{cor:spatial} compete: the first one requires many groups, the second one a large separation $D$, and in a bounded lot the number of groups that are pairwise $D$ apart is limited. If $\mathbb H\subset\R^3$ and $d$ is the Euclidean distance, the open balls of radius $D/2$ centered at one sampling point of each group are disjoint and contained in the $D/2$-neighborhood $\mathbb H^{D/2}$ of $\mathbb H$, hence
	\[
	K\le\frac{\operatorname{Vol}(\mathbb H^{D/2})}{\pi D^3/6}.
	\]
	Thus a small error probability can only be certified if the lot is large compared with the distance at which the spatial dependence becomes negligible. For moderate $K$, bound (ii) is preferable because of the large constants in Lemma~\ref{lem:Rio61}: for $\Delta=\frac14$ it gives $\P^*(|\hat\mu-\mu|\ge\delta)\le\frac4K+16\,\alpha^C_{2B}(D)$, which is below $0.05$ for $K=160$ groups if $\alpha^C_{2B}(D)\le1.5\cdot10^{-3}$, whereas the exponential term of $\Rb_{K,2}$ in (i) falls below $0.025$ only for $K\ge4122$ (and that of the simplified bound only for $K\ge4689$). Moreover, (i) involves $\alpha^C_n(D)$, i.e.\ sets of $n$ points, whereas (ii) only involves $\alpha^C_{2B}(D)$.
\end{remark}

\section*{Declaration of AI use}

During the preparation of this work the authors used ChatGPT (OpenAI) to improve the language, and Claude (Anthropic) to draft proofs and sections, to search and check the literature, and to write a Lean 4 formalization of the main arguments, which was used to check the proofs and is not part of this paper. Specifically, the completion of the proof of Theorem 2.6, the preparation of Sections 3–6, and the Lean formalization were carried out with the assistance of Claude. The authors reviewed and edited all AI-generated content and take full responsibility for the final publication.

\bibliographystyle{siam}
\bibliography{MoM}

\end{document}